\documentclass[11pt,a4paper]{article}
\usepackage[affil-it]{authblk}
\usepackage[utf8]{inputenc}
\usepackage{amsmath, amssymb, mathrsfs, amsthm, amsfonts}
\usepackage{latexsym}
\usepackage[pdftex]{hyperref}
\usepackage{a4wide}
\usepackage{color}
\usepackage{graphicx}

\newtheorem{theorem}{Theorem}[section]

\newtheorem{lemma}[theorem]{Lemma}
\newtheorem{proposition}[theorem]{Proposition}

\theoremstyle{definition}
\newtheorem{definition}[theorem]{Definition}
\theoremstyle{remark}
\newtheorem{example}{Example}

\title{Tutte Polynomial of Ideal Arrangement}

\author{Hery Randriamaro \thanks{Mathematisches Forschungsinstitut Oberwolfach \\ Schwarzwaldstraße 9-11, 77709 Oberwolfach, Germany \\ E-mail: \texttt{hery.randriamaro@outlook.com} \\ This research was supported through the programme "Oberwolfach Leibniz Fellows" by the Mathematisches Forschungsinstitut Oberwolfach in 2017}}

\begin{document}

\maketitle

\begin{abstract}
\noindent The Tutte polynomial is originally a bivariate polynomial enumerating the colorings of a graph and of its dual graph. But it reveals more of the internal structure of the graph like its number of forests, of spanning subgraphs, and of acyclic orientations. In 2007, Ardila extended the notion of Tutte polynomial to the hyperplane arrangements, and computed the Tutte polynomials of the classical root systems for a certain prime power of the first variable. In this article, we compute Tutte polynomials of ideal arrangements. Those arrangements were introduced in 2006 by Sommers and Tymoczko, and are defined for ideals of root systems. For the ideals of the classical root systems, we bring a slight improvement of the finite field method showing that it can applied on any finite field whose cardinality is not a minor of the matrix associated to a hyperplane arrangement. Computing the minor set associated to an ideal of a classical root system permits us particularly to deduce the Tutte polynomials of the classical root systems. For the ideals of the exceptional root systems of type $G_2$, $F_4$, and $E_6$, we use the formula of Crapo.  

\bigskip 


\smallskip

\end{abstract}

\section{Introduction}

\noindent In one of his last papers \cite{Tu3}, Tutte described with these words how in 1954 he became acquainted with the later called Tutte polynomial: ``Playing with my W-functions I obtained a two-variable polynomial from which either the chromatic polynomial or the flow polynomial could be obtained by setting one of the variables equal to zero, and adjusting signs." At the beginning, this polynomial was effectively associated to a graph \cite[3. The dichromate of a graph]{Tu}. But in 2007, Ardila extended the notion of Tutte polynomial to hyperplane arrangement \cite[3. Computing the Tutte polynomial]{Ar}.

\smallskip

\noindent Let $(x_1, \dots, x_n)$ be a variable of the $n$--dimensional space $\mathbb{K}^n$ over the field $\mathbb{K}$, and $a_1, \dots, a_n, b$ in $\mathbb{K}$ such that $(a_1, \dots, a_n) \neq (0, \dots, 0)$. A hyperplane in $\mathbb{K}^n$ is a $(n-1)$--dimensional affine subspace $H := \big\{ (x_1, \dots, x_n) \in \mathbb{K}^n \ |\ a_1 x_1 + \dots + a_n x_n = b \big\}$ that we simply denote by $H = \{a_1 x_1 + \dots + a_n x_n = b\}$. A hyperplane arrangement is a finite set of hyperplanes.\\
Let $\mathcal{A}$ be a hyperplane arrangement, and $\cap \mathcal{A} := \bigcap_{H \in \mathcal{A}}H$. One says that $\mathcal{A}$ is central if $\cap \mathcal{A} \neq \emptyset$. \emph{From now on, every hyperplane arrangement we consider is central}.\\
A subarrangement of a hyperplane arrangement $\mathcal{A}$ in $\mathbb{K}^n$ is a subset of $\mathcal{A}$. The rank function $\mathrm{r}$ is defined for each subarrangement $\mathcal{B}$ of $\mathcal{A}$ by $\mathrm{r}(\mathcal{B}) := n - \dim \cap \mathcal{B}$.

\begin{definition}
the Tutte polynomial of the hyperplane arrangement $\mathcal{A}$ is
$$T_{\mathcal{A}}(x,y) := \sum_{\mathcal{B} \subseteq \mathcal{A}} (x-1)^{\mathrm{r}(\mathcal{A})-\mathrm{r}(\mathcal{B})} (y-1)^{\#\mathcal{B}-\mathrm{r}(\mathcal{B})}.$$
\end{definition}

\paragraph{Characteristic Polynomial.} Let $L(\mathcal{A})$ be the set of nonempty intersections of hyperplanes in $\mathcal{A}$. The elements of $L(\mathcal{A})$ are partially ordered by reverse inclusion with unique minimal element $\mathbb{R}^n$. The characteristic polynomial of $\mathcal{A}$ is $\chi_{\mathcal{A}}(q) := \sum_{E \in L_A} \mu(\mathbb{R}^n, E) q^{\dim(E)}$, where $\mu$ denotes the Möbius function of the lattice $L(\mathcal{A})$. The characteristic polynomial gives important information on the associated hyperplane arrangement. The number of chambers of $\mathcal{A}$ is equal to $(-1)^n \chi_{\mathcal{A}}(-1)$. Another example considers a closed chamber $\bar{C}$ of $\mathcal{A}$. One says that a point of $\mathbb{R}^n$ has a $k$-dimensional projection on $\bar{C}$ if its orthogonal projection onto $\bar{C}$ lies in the relative interior of a $k$-dimensional face of $\bar{C}$. The set of the points for which the projections on $\bar{C}$ is $k$-dimensional forms a cone $K_{C,k}$. The ratio of volume $v_k(C)$ occupied by $K_{C,k}$ is defined by
$$v_k(C) := \frac{|K_{C,k} \cap S^{n-1}|}{|S^{n-1}|},$$
where $|\cdot|$ is the Lebesgue measure, and $S^{n-1}$ the unit sphere. Klivans and Swartz proved that the sum of the $v_k(C)$'s over all chambers $C$ of $\mathcal{A}$ is equal to the absolute value of the coefficient of $q^{\mathrm{r}(\mathcal{A})-n+k}$ in $\chi_{\mathcal{A}}(q)$ \cite[Theorem 5]{KlSw}. The characteristic polynomial of $\mathcal{A}$ is a specialization of its Tutte polynomial by
$$\chi_{\mathcal{A}}(q) = (-1)^{\mathrm{r}(\mathcal{A})} q^{n-\mathrm{r}(\mathcal{A})} T_{\mathcal{A}}(1-q,0).$$

\paragraph{Graphic Arrangement.} A finite simple nonoriented graph $G$ consists on the vertex set $[n]$, and on a subset $E$ of $\binom{[n]}{2}$ as edge set. To the graph $G$ is associated a hyperplane arrangement $\mathcal{A}_G$ in $\mathbb{R}^n$ defined by
$$\mathcal{A}_G := \big\{\{x_i - x_j = 0\}\big\}_{\{i,j\} \in E}.$$
The Tutte polynomial $T_{\mathcal{A}_G}(x,y)$ contains many information on the graph $G$. As examples, $T_{\mathcal{A}_G}(2,1)$ counts the number of forests, $T_{\mathcal{A}_G}(1,1)$ the number of spanning forests, and $T_{\mathcal{A}_G}(1,2)$ the number of spanning subgraphs. Moreover, the correspondence $G \leftrightarrow \mathcal{A}_G$ may be used to pull back results concerning arrangements to results concerning graphs. For example, Zaslavsky's chamber counting theorem can be translated into Stanley's theorem which states that the number of acyclic orientations for graphs is $(-1)^n \chi_{\mathcal{A}_G}(-1)$ \cite[Theorem 2.94]{OrTe}.

\bigskip

\noindent \emph{From now on, we work in the Euclidean space $\mathbb{R}^n$ with inner product $(.,.)$ the usual dot product.} Recall that a reflection $s_u$ associated to a nonzero vector $u$ is a linear map sending $u$ to its negative while fixing pointwise the hyperplane $u^{\perp}$. A root system is a finite set $\Phi$ of nonzero vectors $u$ in $\mathbb{R}^n$ satisfying the conditions 
\begin{itemize}
\item $\Phi \cap \mathbb{R}u = \{u, -u\}$ for all $u \in \Phi$,
\item $s_u(\Phi) = \Phi$ for every $u$ in $\Phi$,
\item $s_u(v)$ has integer coefficients for every $u,v$ in $\Phi$.
\end{itemize}
A root system is irreducible if it cannot be expressed as a disjoint union of two nonempty subsets $\Phi_1 \sqcup \Phi_2$ such that $(u_1,u_2) = 0$ for $u_1$ in $\Phi_1$, and $u_2$ in $\Phi_2$.

\noindent Denote by $\{e_1, \dots, e_n\}$ the standard basis of $\mathbb{R}^n$. There are nine types of irreducible root systems: The four infinite families of root systems associated to the classical Lie algebras
\begin{align*}
& (A_{n-1},\, n \geq 2) \quad \Phi_{A_{n-1}} = \{e_i-e_j\ |\ 1 \leq i \neq j \leq n\},\\
& (B_n,\, n \geq 2) \qquad \Phi_{B_n} = \{\pm e_i \pm e_j\ |\ 1 \leq i < j \leq n\} \cup \{\pm e_i\ |\ i \in [n]\},\\
& (C_n,\, n \geq 2) \qquad \Phi_{C_n} = \{\pm e_i \pm e_j\ |\ 1 \leq i < j \leq n\} \cup \{\pm 2e_i\ |\ i \in [n]\},\\
& (D_n,\, n \geq 4) \qquad \Phi_{D_n} = \{\pm e_i \pm e_j\ |\ 1 \leq i < j \leq n\},
\end{align*}
and the five exceptional root systems
\begin{align*}
& (G_2) \quad \Phi_{G_2} = \big\{\pm (e_i-e_j)\ |\ 1 \leq i < j \leq 3\big\} \cup \big\{\pm (2e_i-e_j-e_k)\ |\ \{i, j, k\} = \{1, 2, 3\}\big\},\\
& (F_4) \quad \Phi_{F_4} = \{\pm e_i \pm e_j\ |\ 1 \leq i < j \leq 4\} \cup \{\pm e_i\ |\ i \in [4]\} \cup \big\{\frac{1}{2}(\pm e_1 \pm e_2 \pm e_3 \pm e_4)\big\},\\
& (E_8) \quad \Phi_{E_8} = \{\pm e_i \pm e_j\ |\ 1 \leq i < j \leq 8\} \cup \big\{\frac{1}{2} \sum_{i=1}^8 \pm e_i\ \text{even number of + signs}\big\},\\
& (E_7) \quad \Phi_{E_7} = \{\pm e_i \pm e_j\ |\ 1 \leq i < j \leq 6\} \cup \big\{\pm(e_7 - e_8)\big\}\\
& \qquad \qquad \qquad \cup \big\{\pm \frac{1}{2}\big(e_7 - e_8 + (\sum_{i=1}^6 \pm e_i\ \text{odd number of + signs)\big)} \big\},\\
& (E_6) \quad \Phi_{E_6} = \{\pm e_i \pm e_j\ |\ 1 \leq i < j \leq 5\}\\
& \qquad \qquad \qquad \cup \big\{\pm \frac{1}{2}\big(e_8 - e_7 - e_6 + (\sum_{i=1}^5 \pm e_i\ \text{odd number of + signs})\big)\big\}.
\end{align*}

\noindent A vector of a root system is called a root. There exist some subsets $\Delta$ of $\Phi$ called simple systems such that $\langle \Delta \rangle = \mathbb{R}^n$ and each root in $\Phi$ is a linear combination of roots in $\Delta$ with coefficients all of the same sign. Fixing a simple system $\Delta$, a positive root system $\Phi^+$ consists of the roots with positive coefficients. We endow $\Phi^+$ with the partial order $\preceq$ defined by $u \preceq v$, provided $v-u$ is a linear combination of positive roots with positive coefficients.

\begin{definition}
An ideal of a root system $\Phi$ is a subset $I$ of $\Phi^+$ satisfying the condition
\begin{center}
If $u \in I$, and $v \in \Phi^+$ so that $u \preceq v$, then $v \in I$.
\end{center}
\end{definition}

\noindent Let $I^c := \Phi^+ \setminus I$ be the complement of an ideal $I$. The ideal arrangement $\mathcal{A}_I$ associated to $I$ is the hyperplane arrangement
$$\mathcal{A}_I := \{u^{\perp}\ |\ u \in I^c\}.$$

\paragraph{Poincaré Polynomial of Ideal.} The height of a root $v = \sum_{u \in \Delta} x_u u$ is $\mathrm{ht}(v) := \sum_{u \in \Delta} x_u$. For an ideal $I$, let $\lambda_i := \#\{u \in I^c\ |\ \mathrm{ht}(u)=i\}$. This gives the height partition $\lambda_1 \geq \lambda_2 \geq \dots$ of $I^c$. With $m_i^I := \#\{\lambda_i\ |\ \lambda_i \geq \lambda_1-i+1\}$, define the dual partition $m_{\lambda_1}^I \geq m_{\lambda_1 -1}^I \geq \dots \geq m_1^I$ of the $\lambda_i$'s. The numbers $m_i^I$ are called the ideal exponents of $I$.\\
A subset $S$ of $I^c$ is said $I^c$-closed if for $u, v$ in $S$, if $u+v \in I^c$ then $u + v \in S$. And $S$ is of Weyl type for $I$ if both $S$ and $I^c \setminus S$ are $I^c$-closed. Denote by $\mathrm{W}_I$ the set of subsets of $I^c$ of Weyl type. Sommers and Tymoczko proved that, for any ideal $I$ of the root systems of type $A_{n-1}, B_n, C_n, G_2, F_4, E_6$, its Poincaré polynomial is \cite[Theorem 4.1]{SoTy}
\begin{equation} \label{Eq1}
\sum_{S \in \mathrm{W}_I} t^{\#S} = \sum_{i=1}^{\lambda_1} (1+t+t^2+ \dots+ t^{m_i^I}).
\end{equation}
A parabolic subsystem is a subset $\Phi' \varsubsetneq \Phi$ such that there exists a subset $\Delta' \varsubsetneq \Delta$ with the property $\Phi' = \langle \Delta' \rangle \cap \Phi$. Röhrle showed another condition for the ideal $I$ to satisfy (\ref{Eq1}) \cite[Theorem 1.26, Theorem 1.27]{Ro}: Suppose that the ideal $I$ of the root system $\Phi$ satisfies one of the following conditions
\begin{itemize}
\item[(i)] $\mathcal{A}_I$ is reducible,
\item[(ii)] $\mathcal{A}_I$ is irreducible, and there exists a maximal parabolic subsystem $\Phi_0$ of $\Phi$ such that, with $\Phi_0^c = \Phi \setminus \Phi_0$, $\Phi_0^c \cap I^c \neq \emptyset$, $\Phi_0^c \cap I^c$ is linearly ordered, and for any $u \neq v$ in $\Phi_0^c \cap I^c$, there is $w$ in $\Phi_0^+$ so that $u$, $v$, and $w$ are linealy dependent.
\end{itemize}
Suppose that for every proper parabolic subsystem of $\Phi$, the Poincaré polynomials of all ideals factor as in (\ref{Eq1}). Then the Poincaré polynomial of $I$ also factors as in (\ref{Eq1}). 

\paragraph{Inductive Freeness of Ideal Arrangement.} Denote the polynomial algebra $\mathbb{R}[x_1, \dots, x_n]$ by $S$. A linear map $\theta: S \rightarrow S$ is a derivation if, for $f,g \in S$, $\theta(fg) = f \theta(g) + g \theta(f)$. Denote by $\mathrm{Der}(S)$ the $S$--module of derivations of $S$. $Q_{\mathcal{A}}$ being the defining polynomial of $\mathcal{A}$, the $S$--submodule $\mathrm{D}(\mathcal{A}) := \{\theta \in \mathrm{Der}(S)\ |\ \theta(Q_{\mathcal{A}}) \in Q_{\mathcal{A}} S\}$ of $\mathrm{Der}(S)$ is the module of $\mathcal{A}$--derivations. Recall that $\mathcal{A}$ is said free if $\mathrm{D}(\mathcal{A})$ is a free $S$--module. Sommers and Tymoczko showed that $\mathcal{A}_I$ is free if the root system is associated to $A_{n-1}$, $B_n$, $C_n$, or $G_2$ \cite[Theorem 11.1]{SoTy}.\\
Let $\varnothing_n$ be the empty arrangement of $\mathbb{R}^n$. The class $\mathcal{IF}$ of inductively free arrangements is the smallest class of hyperplane arrangements satisfying
\begin{itemize}
\item[$(1)$] $\varnothing_n \in \mathcal{IF}$ for $n \geq 0$,
\item[$(2)$] if there exists $H \in \mathcal{A}$ such that $\mathcal{A}^H \in \mathcal{IF}$, $\mathcal{A} \setminus \{H\} \in \mathcal{IF}$, and $\exp \mathcal{A}^H \subseteq \exp \mathcal{A} \setminus \{H\}$, then $\mathcal{A} \in \mathcal{IF}$.
\end{itemize}
Hultman proved that the ideal arrangements associated to the root systems of $A_{n-1}$, $B_n$, $C_n$, and $G_2$ are inductively free \cite[Theorem 6.6, Theorem 7.1]{Hul}. Röhrle proved that the ideal arrangements of type $D_n$ are inductively free \cite[Theorem 1.7]{Ro}. He showed as well that if $I$ is an ideal of a root system $\Phi$, and $I$ satisfies one of the following conditions
\begin{itemize}
\item[(i)] $\mathcal{A}_I$ is reducible,
\item[(ii)] $\mathcal{A}_I$ is irreducible, and there exists a maximal parabolic subsystem $\Phi_0$ of $\Phi$ such that $\Phi_0^c \cap I^c \neq \emptyset$, $\Phi_0^c \cap I^c$ is linearly ordered, and for any $u \neq v$ in $\Phi_0^c \cap I^c$, there is $w$ in $\Phi_0^+$ so that $u$, $v$, and $w$ are linealy dependent,
\item[(iii)] $I$ is composed only of the highest root of $\Phi^+$,
\end{itemize}
and each ideal arrangement of a proper parabolic subsystem is inductively free, then $\mathcal{A}_I$ is inductively free with the nonzero exponents given by the ideal exponents $m_i^I$ of $I$ with the possible exception when the root system is of type $E_8$ and $I$ is one of $4545$ ideals \cite[Theorem~1.9, Theorem~1.13, Theorem~1.14, Theorem~1.15]{Ro}.

\bigskip

\noindent \textbf{We compute, for the four infinite families of root systems, and for the exceptional root systems $\mathbf{G_2}$, $\mathbf{F_4}$, and $\mathbf{E_6}$, the Tutte polynomials of their ideal arrangements.}

\noindent For the root systems of types $A_{n-1}$, $B_n$, $C_n$, and $D_n$, we use a simple transformation of the Tutte polynomial, called coboundary polynomial of a hyperplane arrangement.

\begin{definition}
The coboundary polynomial of a hyperplane arrangement $\mathcal{A}$ is
$$\bar{\chi}_{\mathcal{A}}(q,t) := \sum_{\mathcal{B} \subseteq \mathcal{A}} q^{\mathrm{r}(\mathcal{A})-\mathrm{r}(\mathcal{B})} (t-1)^{\#\mathcal{B}}.$$
\end{definition}

\noindent Since $T_{\mathcal{A}}(x,y) = \frac{1}{(y-1)^{\mathrm{r}(\mathcal{A})}}\ \bar{\chi}_{\mathcal{A}}\big((x-1)(y-1), y \big)$, computing the coboundary polynomial of a hyperplane arrangement is equivalent to computing its Tutte polynomial.

\noindent Recall that the hyperplane arrangements generated by the classical root systems are
\begin{itemize}
\item $\mathcal{A}_{A_{n-1}} = \big\{x_i - x_j = 0\big\}_{1 \leq i < j \leq n}$,
\item $\mathcal{A}_{B_n} = \mathcal{A}_{C_n} = \big\{x_i \pm x_j = 0\big\}_{1 \leq i < j \leq n} \cup \big\{x_i = 0\big\}_{i \in [n]}$,
\item and $\mathcal{A}_{D_n} = \big\{x_i \pm x_j = 0\big\}_{1 \leq i < j \leq n}$.
\end{itemize} 
Using the finite field method, Ardila proved that for all powers of a large enough prime $q$,
\begin{align*}
1 + q \sum_{n \in \mathbb{N}^*} \bar{\chi}_{\mathcal{A}_{A_n}}(q,t) \frac{x^n}{n!} &\ =\ \Big( \sum_{n \in \mathbb{N}} t^{\binom{n}{2}} \frac{x^n}{n!} \Big)^q \ & \text{\cite[Theorem 4.1]{Ar}},\\
\sum_{n \in \mathbb{N}} \bar{\chi}_{\mathcal{A}_{B_n}}(q,t) \frac{x^n}{n!} &\ =\ \Big( \sum_{n \in \mathbb{N}} 2^n t^{\binom{n}{2}} \frac{x^n}{n!} \Big)^{\frac{q-1}{2}} \Big( \sum_{n \in \mathbb{N}} t^{n^2} \frac{x^n}{n!} \Big) \ & \text{\cite[Theorem 4.2]{Ar}},\\
\sum_{n \in \mathbb{N}} \bar{\chi}_{\mathcal{A}_{D_n}}(q,t) \frac{x^n}{n!} &\ =\ \Big( \sum_{n \in \mathbb{N}} 2^n t^{\binom{n}{2}} \frac{x^n}{n!} \Big)^{\frac{q-1}{2}} \Big( \sum_{n \in \mathbb{N}} t^{n(n-1)} \frac{x^n}{n!} \Big) \ & \text{\cite[Theorem 4.3]{Ar}}.
\end{align*}

\noindent In Section \ref{SeRed}, we propose a slight improvement of the finite field method, on one side, showing that it can be applied for any $q$ not in the minor set of a vector set. On the other side, we prove that the minor set of the matrix associated to $\Phi_{A_{n-1}}^+$ is $\{0, \pm 1\}$, while the minor set of the matrix associated to $\Phi_{B_n}^+$ is $\{0,\, \pm 2^0,\, \pm 2^1,\, \dots,\, \pm 2^{\lfloor \frac{n}{2} \rfloor}\}$. Those permit us to compute the Tutte polynomials of $\mathcal{A}_{A_{n-1}}$, $\mathcal{A}_{B_n}$, and $\mathcal{A}_{D_n}$ undermentioned.

\noindent For a positive integer $n$, let $\mathrm{p}(n)$ be the $n^{\text{th}}$ prime number with $\mathrm{p}(0)=2$, and define the polynomial 
$$\mathrm{L}_n(x,y,i) := \prod_{\substack{j \in [n+1] \\ j \neq i}} \frac{(x-1)(y-1) - \mathrm{p}(j)}{\big(\mathrm{p}(i) - \mathrm{p}(j)\big) (y-1)^n}.$$
Moreover, let $\mathrm{Par}(n)$ be the set of ordered partitions of $n$.

\begin{theorem} \label{ThCl}
For an integer $n \geq 2$, let $X_m$ be one of the types in $\{A_{n-1}, B_n, D_n\}$. Then, the Tutte polynomial of $\mathcal{A}_{X_m}$ is 
$$T_{\mathcal{A}_{X_m}}(x,y) = \sum_{k=0}^{\#\Phi_{X_m}^+} \sum_{i \in [m+1]} \mathrm{L}_m(x,y,i)\times [t^k]\bar{\chi}_{\mathcal{A}_{X_m}}\big(\mathrm{p}(i),t\big) \times y^k,$$
\begin{align*}
\text{with}\quad & \bar{\chi}_{\mathcal{A}_{A_{n-1}}}\big(\mathrm{p}(i),t\big) = \sum_{(a_1, \dots, a_u) \in \mathrm{Par}(n)} \binom{\mathrm{p}(i)}{u} \binom{n}{a_1, \dots, a_u} \frac{t^{\sum_{k=1}^u \binom{a_k}{2}}}{\mathrm{p}(i)}, \\
& \bar{\chi}_{\mathcal{A}_{B_n}}\big(\mathrm{p}(i),t\big) = \sum_{a=0}^n \sum_{(b_1, \dots, b_u) \in \mathrm{Par}(n-a)} \binom{\frac{\mathrm{p}(i)-1}{2}}{u}  \binom{n-a}{b_1, \dots, b_u} 2^{\sum_{k=1}^u b_k} t^{a^2 + \sum_{k=1}^{u} \binom{b_k}{2}}, \\
& \bar{\chi}_{\mathcal{A}_{D_n}}\big(\mathrm{p}(i),t\big) = \sum_{a=0}^n \sum_{(b_1, \dots, b_u) \in \mathrm{Par}(n-a)} \binom{\frac{\mathrm{p}(i)-1}{2}}{u}  \binom{n-a}{b_1, \dots, b_u} 2^{\sum_{k=1}^u b_k} t^{a(a-1) + \sum_{k=1}^{u} \binom{b_k}{2}}. 
\end{align*}
\end{theorem}

\noindent In Section \ref{SeYo}, we introduce a topology on the shifted Young diagrams, and recall the shifted Young diagrams associated the classical root systems. The open sets of the later diagrams correspond, in fact, to the ideal arrangements. The definition of a full ideal arrangement is particularly given.

\smallskip

\noindent In Section \ref{SePa}, we define the signature $s_I(i)$ of an integer $i$ in accordance with an ideal $I$, and the partition $A^{(1)}| \dots |A^{(r)}|B^{(1)}| \dots |B^{(s)}$ of $[n]$ in accordance with $I$. We need them to compute the coboundary polynomial of an ideal arrangement using the finite field method. Remark that the signature $s_I$ depends on the ideal $I$.

\smallskip

\noindent In Section \ref{SeCob}, we can finally compute the coboundary polynomial of an ideal arrangement associated to a classical root system. Then, we deduce the Tutte polynomials of the ideal arrangements associated to $\Phi_{A_{n-1}}$, $\Phi_{B_n}$, $\Phi_{C_n}$, and $\Phi_{D_n}$ undermentioned. Remark that, since the Tutte polynomial of an ideal arrangement is equal the product of the Tutte polynomials associated to its connected ideal subarrangements, we just need to consider the full connected ideal arrangements.

\begin{theorem}
Let $\mathcal{A}_I$ be a full connected ideal arrangement of $\Phi_{A_{n-1}}$, with associated partition $A^{(1)}| \dots |A^{(r)}$, and let $R^{(u)} = \big\{v \in \{u+1, \dots, r\}\ |\ s_I(A^{(u)}) \cap s_I(A^{(v)}) \neq \emptyset \big\}$. Then, the Tutte polynomial of $\mathcal{A}_I$ is
\begin{align*}
& T_{\mathcal{A}_I}(x,y) = \sum_{k=0}^{\#I^c} \sum_{i \in [n]} \mathrm{L}_{n-1}(x,y,i)\times [t^k]\bar{\chi}_{\mathcal{A}_I}\big(\mathrm{p}(i),t\big) \times y^k, \\
& \text{with} \ \bar{\chi}_{\mathcal{A}_I}\big(\mathrm{p}(i),t\big) = \sum_{\substack{a_1^{(1)} + \dots + a_{\mathrm{p}(i)}^{(1)}\ =\ \# A^{(1)} \\ \vdots \\ a_1^{(r)} + \dots + a_{\mathrm{p}(i)}^{(r)}\ =\ \# A^{(r)}}} \prod_{u=1}^r \binom{\# A^{(u)}}{a_1^{(u)}, \dots, a_{\mathrm{p}(i)}^{(u)}} \frac{t^{\sum_{s=1}^{\mathrm{p}(i)} \binom{a_s^{(u)}}{2} + a_s^{(u)} \sum_{v \in R^{(u)}} a_s^{(v)}}}{\mathrm{p}(i)}.
\end{align*}
\end{theorem}

\begin{theorem}
Let $\mathcal{A}_I$ be a full connected ideal arrangement of $\Phi_{B_n}$ or $\Phi_{C_n}$, with associated partition $A^{(1)}| \dots |A^{(r)}|B^{(1)}| \dots |B^{(s)}$, and $R^{(u)} = \big\{l \in \{u+1, \dots, r\}\ |\ s_I(A^{(u)}) \cap s_I(A^{(l)}) \neq \emptyset \big\}$, $R_A^{(v)} = \big\{l \in [r]\ |\ s_I(B^{(v)}) \cap s_I(A^{(l)}) \neq \emptyset \big\}$, $S^{(v)} = \big\{h \in [v-1]\ |\ s_I(B^{(v)}) \cap s_I(B^{(h)}) \neq \emptyset \big\}$, $R_0 = \big\{l \in [r]\ |\ s_I(0) \cap s_I(A^{(l)}) \neq \emptyset \big\}$, and $S_0 = \big\{h \in [s]\ |\ s_I(0) \cap s_I(B^{(h)}) \neq \emptyset \big\}$.\\
Then, the Tutte polynomial of $\mathcal{A}_I$ is
\begin{align*}
& T_{\mathcal{A}_I}(x,y) = \sum_{k=0}^{\#I^c} \sum_{i \in [n+1]} \mathrm{L}_n(x,y,i)\times [t^k]\bar{\chi}_{\mathcal{A}_I}\big(\mathrm{p}(i),t\big) \times y^k \\
& \text{with}\ \bar{\chi}_{\mathcal{A}_I}\big(\mathrm{p}(i),t\big) = \sum_{\substack{a_0^{(1)} + \dots + a_{\mathrm{p}(i)-1}^{(1)} = \# A^{(1)} \\ \vdots \\ a_0^{(r)} + \dots + a_{\mathrm{p}(i)-1}^{(r)} = \# A^{(r)} \\ b_0^{(1)} + \dots + b_{\mathrm{p}(i)-1}^{(1)} = \# B^{(1)} \\ \vdots \\ b_0^{(s)} + \dots + b_{\mathrm{p}(i)-1}^{(s)} = \# B^{(s)}}} \prod_{u=1}^r \binom{\# A^{(u)}}{a_0^{(u)}, \dots, a_{\mathrm{p}(i)-1}^{(u)}} \prod_{v=1}^s \binom{\# B^{(v)}}{b_0^{(v)}, \dots, b_{\mathrm{p}(i)-1}^{(v)}}t^{\mathrm{f}_B(u,v)},
\end{align*}
\begin{align*}
\text{and}\ \mathrm{f}_B(u,v) = & \sum_{q=0}^{\mathrm{p}(i)-1} \Big( \binom{a_q^{(u)}}{2} + a_q^{(u)} \sum_{l \in R^{(u)}} a_q^{(l)} \Big) + 2 b_0^{(v)} \Big( \frac{b_0^{(v)}-1}{2} + \sum_{l \in R_A^{(v)}} a_0^{(l)} + \sum_{h \in S^{(v)}} b_0^{(h)} \Big) \\
& + \sum_{q=1}^{\mathrm{p}(i)-1} b_q^{(v)} \Big( \frac{b_q^{(v)}-1}{2} + \sum_{l \in R_A^{(v)}} \big( a_q^{(l)} + a_{\mathrm{p}(i)-q}^{(l)} \big) + \sum_{h \in S^{(v)}} \big( b_q^{(h)} + b_{\mathrm{p}(i)-q}^{(h)} \big) \Big) \\
& + \sum_{q=1}^{\frac{\mathrm{p}(i)-1}{2}} b_q^{(v)} \times b_{\mathrm{p}(i)-q}^{(v)} + \sum_{l \in R_0} a_0^{(l)} + \sum_{h \in S_0} b_0^{(h)}.
\end{align*}
\end{theorem}

\begin{theorem}
Let $\mathcal{A}_I$ be a full connected ideal arrangement of $\Phi_{D_n}$, with associated partition $A^{(1)}| \dots |A^{(r)}|B^{(1)}| \dots |B^{(s)}$, and let $R^{(u)} = \big\{l \in \{u+1, \dots, r\}\ |\ s_I(A^{(u)}) \cap s_I(A^{(l)}) \neq \emptyset \big\}$, $R_A^{(v)} = \big\{l \in [r]\ |\ s_I(B^{(v)}) \cap s_I(A^{(l)}) \neq \emptyset \big\}$, and $S^{(v)} = \big\{h \in [v-1]\ |\ s_I(B^{(v)}) \cap s_I(B^{(h)}) \neq \emptyset \big\}$.\\
Then, the Tutte polynomial of $\mathcal{A}_I$ is
\begin{align*}
& T_{\mathcal{A}_I}(x,y) = \sum_{k=0}^{\#I^c} \sum_{i \in [n+1]} \mathrm{L}_n(x,y,i)\times [t^k]\bar{\chi}_{\mathcal{A}_I}\big(\mathrm{p}(i),t\big) \times y^k \\
& \text{with}\ \bar{\chi}_{\mathcal{A}_I}\big(\mathrm{p}(i),t\big) = \sum_{\substack{a_0^{(1)} + \dots + a_{\mathrm{p}(i)-1}^{(1)} = \# A^{(1)} \\ \vdots \\ a_0^{(r)} + \dots + a_{\mathrm{p}(i)-1}^{(r)} = \# A^{(r)} \\ b_0^{(1)} + \dots + b_{\mathrm{p}(i)-1}^{(1)} = \# B^{(1)} \\ \vdots \\ b_0^{(s)} + \dots + b_{\mathrm{p}(i)-1}^{(s)} = \# B^{(s)}}} \prod_{u=1}^r \binom{\# A^{(u)}}{a_0^{(u)}, \dots, a_{\mathrm{p}(i)-1}^{(u)}} \prod_{v=1}^s \binom{\# B^{(v)}}{b_0^{(v)}, \dots, b_{\mathrm{p}(i)-1}^{(v)}}t^{\mathrm{f}_D(u,v)},
\end{align*}
\begin{align*}
\text{and}\ \mathrm{f}_D(u,v) = & \sum_{q=0}^{\mathrm{p}(i)-1} \Big( \binom{a_q^{(u)}}{2} + a_q^{(u)} \sum_{l \in R^{(u)}} a_q^{(l)} \Big) + 2 b_0^{(v)} \Big( \frac{b_0^{(v)}-1}{2} + \sum_{l \in R_A^{(v)}} a_0^{(l)} + \sum_{h \in S^{(v)}} b_0^{(h)} \Big) \\
& + \sum_{q=1}^{\mathrm{p}(i)-1} b_q^{(v)} \Big( \frac{b_q^{(v)}-1}{2} + \sum_{l \in R_A^{(v)}} \big( a_q^{(l)} + a_{\mathrm{p}(i)-q}^{(l)} \big) + \sum_{h \in S^{(v)}} \big( b_q^{(h)} + b_{\mathrm{p}(i)-q}^{(h)} \big) \Big) \\
& + \sum_{q=1}^{\frac{\mathrm{p}(i)-1}{2}} b_q^{(v)} \times b_{\mathrm{p}(i)-q}^{(v)}.
\end{align*}
\end{theorem}

\noindent In Section \ref{SeEx}, we show how to compute the Tutte polynomial of an ideal arrangement of $\Phi_{G_2}^+$, $\Phi_{F_4}^+$, and $\Phi_{E_6}^+$. For most ideals $I$, one can not directly use the definition of the Tutte polynomial for the computing. Indeed, the algorithm would implement $\binom{\#I^c}{k}$ sets of cardinality $k$, where $k$ varies from $1$ to $\#I^c$, so that the space and time complexity would exceed the capacity of our computer. That is why we use the formula of Crapo \cite[Theorem~2.32]{DePr} which reduces the algorithm implementation on $\binom{\#I^c}{\mathrm{r}(\mathcal{A}_I)}$ sets of cardinality $\mathrm{r}(\mathcal{A}_I)$.

\bigskip

\noindent \emph{The author would like to thank Gerhard Röhrle to have initiated him to ideal arrangements.}

\section{Correct Reduction} \label{SeRed}

\noindent The finite field method reduces the coboundary polynomial computing for certain prime powers to a counting problem. We propose a slight improvement of that method which permits to determine the prime numbers for which it can be used. Then, we compute the minor sets of the matrices associated to the classical root systems. The finite field method is, in fact, valid for prime numbers not included in those minor sets. So, we get the valid prime numbers to use that method and the interpolation formula of Lagrange in order to compute the coboundary polynomial of an ideal arrangement. By the way, we complete the calculations of Ardila to obtain the Tutte polynomials associated to the classical root systems.

\begin{definition}
Two hyperplane arrangements $\mathcal{A}$ and $\mathcal{B}$ are isomorphic if there is an order preserving bijection between the lattices $L(\mathcal{A})$ and $L(\mathcal{B})$.
\end{definition}

\noindent A hyperplane arrangement whose coefficients lie in $\mathbb{Z}$ is called a $\mathbb{Z}$--arrangement. Furthermore, for a prime number $p$, the finite field composed by the integers modulo $p$ is denoted by $\mathbb{F}_p$.

\begin{definition}
Take a $\mathbb{Z}$--arrangement $\mathcal{A}$ in $\mathbb{R}^n$, and a prime number $p$. For a hyperplane $H = \{a_1 x_1 + \dots + a_n x_n = b\}$ in $\mathbb{R}^n$, define the set $\bar{H} = \{\bar{a}_1 \bar{x}_1 + \dots + \bar{a}_n \bar{x}_n = \bar{b}\}$ in $\mathbb{F}_p^n$. One says that $\mathcal{A}$ reduces correctly over $\mathbb{F}_p$ if
\begin{itemize}
\item for every hyperplane $H$ in $\mathcal{A}$, $\bar{H}$ is a hyperplane in $\mathbb{F}_p^n$,
\item and, if we define $\bar{\mathcal{A}} := \{\bar{H}\ |\ H \in \mathcal{A}\}$, $\mathcal{A}$ and $\bar{\mathcal{A}}$ are isomorphic.
\end{itemize}
\end{definition}

\noindent Let $U = \{u_1, \dots, u_m\}$ be a vector set in $\mathbb{R}^n$. Define its associated matrix by $$\mathsf{M}_U := (u_i)_{i \in [m]}\ \text{where}\ u_i\ \text{is the $i^{\text{th}}$ row}.$$

\noindent Denote by $\mathrm{Min}(U)$ the minor set of the matrix $\mathsf{M}_U$.

\begin{lemma} \label{LeRe}
Take a $\mathbb{Z}$--vector set $U$ in $\mathbb{R}^n$, and a prime number $p$. Then, the central $\mathbb{Z}$--arrangement $\mathcal{A} = \{u^{\perp}\}_{u \in U}$ reduces correctly over $\mathbb{F}_p$ if $p \notin \big\{|i|\ |\ i \in \mathrm{Min}(U)\big\}$.
\end{lemma}

\begin{proof}
Since $p \notin \big\{|i|\ |\ i \in \mathrm{Min}(U)\big\}$, for every subset $\mathcal{B}$ of $\mathcal{A}$, $\mathrm{r}(\mathcal{B}) = \mathrm{r}(\bar{\mathcal{B}})$. That implies that
\begin{itemize}
\item $\bar{H}$ is a hyperplane in $\mathbb{F}_p^n$ for every $H$ in $\mathcal{A}$,
\item $\forall \mathcal{B},\mathcal{B}' \subseteq \mathcal{A}:\, \cap \mathcal{B} \neq \cap \mathcal{B}' \Rightarrow \cap \bar{\mathcal{B}} \neq \cap \bar{\mathcal{B}}'$,
\item $\forall \mathcal{B},\mathcal{B}' \subseteq \mathcal{A}:\, \cap \mathcal{B} \subseteq \cap \mathcal{B}' \Rightarrow \cap \bar{\mathcal{B}} \subseteq \cap \bar{\mathcal{B}}'$.
\end{itemize}
So, the function from $L(\mathcal{A})$ to $L(\bar{\mathcal{A}})$, mapping $\cap \mathcal{B}$ to $\cap \bar{\mathcal{B}}$, is an order preserving bijection. 
\end{proof}

\noindent For a hyperplane arrangement $\bar{\mathcal{A}}$, and a vector $\bar{x}$ in $\mathbb{F}_p^n$, define the hyperplane arrangement
$$\bar{\mathcal{A}}(\bar{x}) := \{\bar{H} \in \bar{\mathcal{A}}\ |\ \bar{x} \in \bar{H}\}.$$
To compute the coboundary polynomial, we use the finite field method based on this theorem.

\begin{theorem} \label{ThFi}
Consider a $\mathbb{Z}$--vector set $U$ in $\mathbb{R}^n$, and its associated central arrangement $\mathcal{A} = \{u^{\perp}\}_{u \in U}$. Let $p$ be a prime number in $\mathbb{N} \setminus \mathrm{Min}(U)$, and $\bar{\mathcal{A}}$ the reduction of $\mathcal{A}$ over $\mathbb{F}_p$. Then, $$p^{n-\mathrm{r}(\mathcal{A})} \bar{\chi}_{\mathcal{A}}(p,t) = \sum_{\bar{x} \in \mathbb{F}_p^n} t^{\#\bar{\mathcal{A}}(\bar{x})}.$$
\end{theorem}

\begin{proof}
Remark that
\begin{itemize}
\item[(R1)] if $\bar{V}$ is a $m$-dimensional subspace of $\mathbb{F}_p^n$, then $\#\bar{V} = p^m$,
\item[(R2)] and, for a strictly positive integer $m$, we have $\sum_{I \subseteq [m]}t^{|I|} = (t+1)^m$.
\end{itemize}
Then,
\begin{align*}
p^{n-\mathrm{r}(\mathcal{A})} \bar{\chi}_{\mathcal{A}}(p,t) & = \sum_{\mathcal{B} \subseteq \mathcal{A}} p^{n-\mathrm{r}(\mathcal{B})} (t-1)^{\#\mathcal{B}} \\
& = \sum_{\mathcal{B} \subseteq \mathcal{A}} p^{\dim \cap \mathcal{B}} (t-1)^{\#\mathcal{B}}
= \sum_{\bar{\mathcal{B}} \subseteq \bar{\mathcal{A}}} p^{\dim \cap \bar{\mathcal{B}}} (t-1)^{\#\bar{\mathcal{B}}} \\
& = \sum_{\bar{\mathcal{B}} \subseteq \bar{\mathcal{A}}} \#\cap \bar{\mathcal{B}}\ (t-1)^{\#\bar{\mathcal{B}}} \quad \text{(R1)} \\
& = \sum_{\bar{\mathcal{B}} \subseteq \bar{\mathcal{A}}} \sum_{\bar{x} \in \cap \bar{\mathcal{B}}} (t-1)^{\#\bar{\mathcal{B}}}
= \sum_{\bar{x} \in \mathbb{F}_p^n} \sum_{\bar{\mathcal{B}} \subseteq \bar{\mathcal{A}}(\bar{x})} (t-1)^{\#\bar{\mathcal{B}}} \\
& = \sum_{\bar{x} \in \mathbb{F}_p^n} t^{\#\bar{\mathcal{A}}(\bar{x})} \quad \text{(R2)}.
\end{align*}
\end{proof}

\noindent Recall that
$$\Phi_{A_{n-1}}^+ = \{e_i-e_j\ |\ 1 \leq i < j \leq n\}\quad \text{and}\quad \Phi_{B_n}^+ = \{e_i \pm e_j\ |\ 1 \leq i < j \leq n\} \cup \{e_i\ |\ i \in [n]\}.$$
Now, we compute the minor sets of the matrices associated to $\Phi_{A_{n-1}}^+$, and $\Phi_{B_n}^+$.

\noindent Denote by $[\pm n]$ the set $\{-n, \dots, -1, 1, \dots, n\}$. For $i \in [n-1]$, and $r \in [\pm n]$ with $i < |r|$, define the vector $$d(i,r) := e_i + \mathrm{sgn}(r)\,e_{|r|},\ \text{where $\mathrm{sgn}$ is the signum function}.$$
Denote by $\mathscr{D}_n$ the set of square matrices $\mathsf{M} = (u_l)_{l \in [n]}$ of order $n$ such that $u_l = d(i_l,r_l)$.

\smallskip

\noindent Take a matrix $\mathsf{M} = (u_l)_{l \in [n]}$ of $\mathscr{D}_n$ with $\det \mathsf{M} \neq 0$. This determinant condition implies that there is at most two rows $u_k = d(i_k,r_k)$ and $u_m = d(i_m,r_m)$ such that $i_k = i_m$ and $|r_k| = |r_m|$. 

\paragraph*{Algorithm D1.} Suppose first that $\mathsf{M}$ does not contain two such rows. We transform $\mathsf{M}$ into the matrix $(1) \oplus \mathsf{M}'$, where $\mathsf{M}' \in \mathscr{D}_{n-1}$.\\
\textbf{D1-1.} We begin with $\mathsf{M} = \big(d(i_l,r_l)\big)_{l \in [n]}$.\\
\textbf{D1-2.} Denote by $\mathsf{T}_{i,j}$ the elementary matrix which switches the $i^{\text{th}}$ row with the $j^{\text{th}}$ row. Let $k = \min \{l \in [n]\ |\ i_l = 1\}$.
$$\text{If}\ k \neq 1,\ \text{Then set}\ \mathsf{M} \leftarrow \mathsf{T}_{1,k} \cdot \mathsf{M}.$$
\textbf{D1-3.} Denote by $\mathsf{R}_{i,j}(m)$ the elementary matrix which adds the $i^{\text{th}}$ row multiplied by a scalar $m$ to the $j^{\text{th}}$ row. Set
$$\mathsf{M} \leftarrow \prod_{\substack{2 \leq l \leq n \\ i_l = 1}} \mathsf{R}_{1,l}(-1) \cdot \mathsf{M}.$$
At this step, if $\mathsf{M} = (u_l)_{l \in [n]}$, we have
\begin{itemize}
\item $\forall l \in [n]:\ u_l = d(i_l,r_l)\ \text{or}\ u_l = - d(i_l,r_l)$,
\item if $2 \leq l$, then $i_l \geq 2$.
\end{itemize}
\textbf{D1-4.} Denote by $\mathsf{D}_i(m)$ the elementary matrix which multiplies all elements on the $i^{\text{th}}$ row by a nonzero scalar $m$. Set
$$\mathsf{M} \leftarrow \prod_{\substack{2 \leq l \leq n \\ u_l = - d(i_l,r_l)}} \mathsf{D}_l(-1) \cdot \mathsf{M}.$$
\textbf{D1-5.} Denote by $\mathsf{C}_{i,j}(m)$ the elementary matrix which adds the $i^{\text{th}}$ column multiplied by a scalar $m$ to the $j^{\text{th}}$ column.
$$\text{If}\ \mathrm{sgn}(r_1) = -1\quad \text{Then set}\ \mathsf{M} \leftarrow \mathsf{M} \cdot \mathsf{C}_{1,|r_1|}(1)\quad \text{Else set}\ \mathsf{M} \leftarrow \mathsf{M} \cdot \mathsf{C}_{1,|r_1|}(-1).$$
\textbf{D1-6.} Return $\mathsf{M}$.

\begin{example} Applying Algorithm D1 on $\begin{pmatrix}
0 & 1 & -1 & 0 & 0 \\ 1 & 0 & -1 & 0 & 0 \\ 1 & 0 & 0 & 1 & 0 \\ 0 & 0 & 1 & 0 & 1 \\ 1 & 0 & 0 & 0 & -1
\end{pmatrix}$, we obtain $\begin{pmatrix}
1 & 0 & 0 & 0 & 0 \\ 0 & 1 & -1 & 0 & 0 \\ 0 & 0 & 1 & 1 & 0 \\ 0 & 0 & 1 & 0 & 1 \\ 0 & 0 & 1 & 0 & -1
\end{pmatrix}$.
\end{example}

\paragraph*{Algorithm D2.} Suppose now that $\mathsf{M}$ contains two such rows with $r_k < 0$. We transform $\mathsf{M}$ into a matrix $\mathsf{M}' = (u_l')_{l \in [n]}$ of $\mathscr{D}_n$, with $u_1' = d(1,-2)$ and $u_2' = d(1,2)$.\\
\textbf{D2-1.} We begin with $\mathsf{M} = \big(d(i_l,r_l)\big)_{l \in [n]}$.\\
\textbf{D2-2.} If $m \neq 1$ $\quad$ Then set $\mathsf{M} \leftarrow \mathbf{T}_{1,k} \cdot \mathsf{M}$ and set $\mathsf{M} \leftarrow \mathsf{T}_{2,m} \cdot \mathsf{M}$ $\quad$ Else,\\
If $k \neq 2$ $\quad$ Then set $\mathsf{M} \leftarrow \mathsf{T}_{2,m} \cdot \mathsf{M}$ and set $\mathsf{M} \leftarrow \mathsf{T}_{1,k} \cdot \mathsf{M}$ $\quad$ Else set $\mathsf{M} \leftarrow \mathsf{T}_{1,2} \cdot \mathsf{M}$.\\
\textbf{D2-3.} Denote by $\mathsf{L}_{i,j}$ the elementary matrix which switches the $i^{\text{th}}$ column with the $j^{\text{th}}$ column. Set
$$\mathsf{M} \leftarrow \mathsf{M} \cdot \mathsf{L}_{1,i_k} \cdot \mathsf{L}_{2,r_m}.$$
\textbf{D2-4.} Set $$\mathsf{M} \leftarrow \prod_{\substack{3 \leq l \leq n \\ u_l = - d(i_l,r_l)}} \mathsf{D}_l(-1) \cdot \mathsf{M}.$$
\textbf{D2-5.} Return $\mathsf{M}$.

\begin{example} Applying Algorithm D2 on $\begin{pmatrix}
1 & 0 & 0 & 1 & 0 \\ 0 & 0 & 1 & -1 & 0 \\ 0 & 1 & 0 & 1 & 0 \\ 0 & 1 & 0 & 0 & -1 \\ 0 & 1 & 0 & -1 & 0
\end{pmatrix}$, we obtain $\begin{pmatrix}
1 & -1 & 0 & 0 & 0 \\ 1 & 1 & 0 & 0 & 0 \\ 0 & 1 & -1 & 0 & 0 \\ 1 & 0 & 0 & 0 & -1 \\ 0 & 1 & 0 & 1 & 0
\end{pmatrix}$.
\end{example}

\begin{lemma} \label{LeDn}
Let $\mathsf{M}$ be a square submatrix of order $m$ of a matrix in $\mathscr{D}_n$. Then,
\begin{itemize}
\item whether $|\det \mathsf{M}| \in \{0, 1\}$,
\item or there exist an integer $m' \in [m]$, and a matrix $\mathsf{M}' \in \mathscr{D}_{m'}$ such that $|\det \mathsf{M}| =  |\det \mathsf{M}'|$.
\end{itemize}
\end{lemma}

\begin{proof}
Suppose that $\det \mathsf{M} \neq 0$. If $\mathsf{M} \in \mathscr{D}_m$, then we are obviously done. Otherwise, there is an integer $i$ of $[m]$ such that the $i^{\text{th}}$ row of $\mathsf{M}$ has entries $0$ everywhere except in the $j^{\text{th}}$ position, where it is $-1$ or $1$. Denoting by $\mathsf{M}^{(i,j)}$ the submatrix of $\mathsf{M}$ obtained by deleting the $i^{\text{th}}$ row and the $j^{\text{th}}$ column of $\mathsf{M}$, we obtain $|\det \mathsf{M}| = |\det \mathsf{M}^{(i,j)}|$.\\
Setting $\mathsf{M} \leftarrow \mathsf{M}^{(i,j)}$, and repeating the same process as long as necessary, whether we end up at $|\det \mathsf{M}| = 1$ at the end, or we come to a nonnegative integer $m' < m$ and a square submatrix $\mathsf{M}'$ of $\mathsf{M}$ such that $\mathsf{M}' \in \mathscr{D}_{m'}$ and $|\det \mathsf{M}| = |\det \mathsf{M}'|$.
\end{proof}

\begin{proposition} \label{PrDn}
Let $\mathsf{M}$ be a square submatrix of a matrix in $\mathscr{D}_n$. Then,
$$|\det \mathsf{M}| \in \{0, 2^0, 2^1, \dots, 2^{\lfloor \frac{n}{2} \rfloor}\}.$$
\end{proposition}

\begin{proof}
Suppose that $\mathsf{M} \in \mathscr{D}_n$ with $\det \mathsf{M} \neq 0$. This condition infers that there is at most two rows $u_k = d(i_k,r_k)$ and $u_m = d(i_m,r_m)$ such that $i_k = i_m$ and $|r_k| = |r_m|$.
\begin{itemize}
\item If $\mathsf{M}$ does not contain two such rows, Algorithm D1 permits us to deduce that there exists a matrix $\mathsf{M}' \in \mathscr{D}_{n-1}$ such that $|\det \mathsf{M}| = |\det \mathsf{M}'|$.
\item Otherwise, Algorithm D2 infers that there is a square submatrix $\mathsf{M}'$ of order $n-2$ of a matrix in $\mathscr{D}_n$ such that $|\det \mathsf{M}| = \begin{vmatrix} 1 & -1 \\ 1 & 1 \end{vmatrix} |\det \mathsf{M}'| = 2 |\det \mathsf{M}'|$.
\end{itemize}
Suppose now that $\mathsf{M}$ is a square submatrix of a matrix in $\mathscr{D}_n$ such that $\det \mathsf{M} \neq 0$. By using Lemma \ref{LeDn}, we deduce after a recursive argument that the value of $|\det \mathsf{M}|$ must belong to the set $\{2^0, 2^1, \dots, 2^{\lfloor \frac{n}{2} \rfloor}\}$.
\end{proof}

\noindent Denote by $\mathscr{A}_n$ the subset of $\mathscr{D}_n$ consisting of the matrices $\mathsf{M} = \big(d(i_l,r_l)\big)_{l \in [n]}$ such that $r_l < 0$.

\begin{lemma} \label{LeAn}
Let $\mathsf{M}$ be a square submatrix of a matrix in $\mathscr{A}_n$. Then, $|\det \mathsf{M}| \in \{0, 1\}$.
\end{lemma}

\begin{proof}
It is known that the dimension of the subspace generated by $\big\{d(i,-j)\big\}_{1 \leq i < j \leq n}$ is $n-1$, as its orthogonal complement is $\langle e_1 + \dots + e_n \rangle$. Therefore, for every $\mathsf{M}$ in $\mathscr{A}_n$, $\det \mathsf{M} = 0$.\\
Now, suppose that $\mathsf{M}$ be a square submatrix of order $m$ of a matrix in $\mathscr{A}_n$. With an argument similar to the proof of Lemma \ref{LeDn}, whether we end up at $|\det \mathsf{M}| = 1$ at the end, or we come to a nonnegative integer $m' \leq m$ and a square submatrix $\mathsf{M}'$ of $\mathsf{M}$ such that $\mathsf{M}' \in \mathscr{A}_{m'}$ and $|\det \mathsf{M}| = |\det \mathsf{M}'| = 0$.
\end{proof}

\noindent We come to the main result of this section.
 
\begin{theorem} \label{ThMi}
Take an integer $n \geq 2$. Then
$$\mathrm{Min}(\Phi_{A_{n-1}}^+) = \{0,\, \pm 1\} \quad \text{and} \quad \mathrm{Min}(\Phi_{B_n}^+) = \{0,\, \pm 2^0,\, \pm 2^1,\, \dots,\, \pm 2^{\lfloor \frac{n}{2} \rfloor}\}.$$
\end{theorem}

\begin{proof}
A minor of $\mathsf{M}_{\Phi_{A_{n-1}}^+}$ is the determinant of a square submatrix of a matrix in $\mathscr{A}_n$. Then, we deduce from Lemma \ref{LeAn} that $\mathrm{Min}(\Phi_{A_{n-1}}^+) = \{0,\, \pm 1\}$.\\
There exists an integer $r \geq n$ such that a minor $t$ of $\mathsf{M}_{\Phi_{B_n}^+}$ is the determinant of a square submatrix of a matrix in $\mathscr{D}_r$. From Lemma \ref{LeDn}, we deduce that
\begin{itemize}
\item whether $|t| \in \{0, 1\}$,
\item or there exist an integer $m \in [n]$, and a matrix $\mathsf{M} \in \mathscr{D}_{m}$ such that $|t| = |\det \mathsf{M}|$.
\end{itemize}
We conclude, using Proposition \ref{PrDn}, that $\mathrm{Min}(\Phi_{B_n}^+) = \{0,\, \pm 2^0,\, \pm 2^1,\, \dots,\, \pm 2^{\lfloor \frac{n}{2} \rfloor}\}$.
\end{proof}

\noindent We deduce the coboundary polynomials of $\mathcal{A}_{A_{n-1}}$, $\mathcal{A}_{B_n}$, and $\mathcal{A}_{D_n}$.

\begin{theorem} \label{ThCoCl}
For an integer $n \geq 2$, let $X_m$ be one of the types in $\{A_{n-1}, B_n, D_n\}$. Then, the coboundary polynomial of $\mathcal{A}_{X_m}$ is 
$$\bar{\chi}_{\mathcal{A}_{X_m}}(q,t) = \sum_{k=0}^{\#\Phi_{X_m}^+} \sum_{i \in [m+1]} \Big(\prod_{\substack{j \in [m+1] \\ j \neq i}} \frac{q - \mathrm{p}(j)}{\mathrm{p}(i) - \mathrm{p}(j)}\Big) \times [t^k]\bar{\chi}_{\mathcal{A}_{X_m}}\big(\mathrm{p}(i),t\big) \times t^k,$$
\begin{align*}
\text{with}\quad & \bar{\chi}_{\mathcal{A}_{A_{n-1}}}\big(\mathrm{p}(i),t\big) = \sum_{(a_1, \dots, a_u) \in \mathrm{Par}(n)} \binom{\mathrm{p}(i)}{u} \binom{n}{a_1, \dots, a_u} \frac{t^{\sum_{k=1}^u \binom{a_k}{2}}}{\mathrm{p}(i)}, \\
& \bar{\chi}_{\mathcal{A}_{B_n}}\big(\mathrm{p}(i),t\big) = \sum_{a=0}^n \sum_{(b_1, \dots, b_u) \in \mathrm{Par}(n-a)} \binom{\frac{\mathrm{p}(i)-1}{2}}{u}  \binom{n-a}{b_1, \dots, b_u} 2^{\sum_{k=1}^u b_k} t^{a^2 + \sum_{k=1}^{u} \binom{b_k}{2}}, \\
& \bar{\chi}_{\mathcal{A}_{D_n}}\big(\mathrm{p}(i),t\big) = \sum_{a=0}^n \sum_{(b_1, \dots, b_u) \in \mathrm{Par}(n-a)} \binom{\frac{\mathrm{p}(i)-1}{2}}{u}  \binom{n-a}{b_1, \dots, b_u} 2^{\sum_{k=1}^u b_k} t^{a(a-1) + \sum_{k=1}^{u} \binom{b_k}{2}}. 
\end{align*}
\end{theorem}

\begin{proof}
The degree of $[t^k]\bar{\chi}_{\mathcal{A}_{X_m}}(q,t)$ in variable $q$ is less than or equal to $m$. From Lemma~\ref{LeRe}, and Theorem~\ref{ThMi}, we know that $\big(\mathrm{p}(i), [t^k]\bar{\chi}_{\mathcal{A}}(\mathrm{p}(i),t)\big)$, with $i \in [m+1]$, are $m+1$ valid data tuples. Then, using the polynomial interpolation of Lagrange, we obtain
$$\bar{\chi}_{\mathcal{A}_{X_m}}(q,t) = \sum_{k=0}^{\#\Phi_{X_m}^+} \sum_{i \in [m+1]} \Big(\prod_{\substack{j \in [m+1] \\ j \neq i}} \frac{q - \mathrm{p}(j)}{\mathrm{p}(i) - \mathrm{p}(j)}\Big) \times [t^k]\bar{\chi}_{\mathcal{A}_{X_m}}\big(\mathrm{p}(i),t\big) \times t^k.$$
The polynomials $\bar{\chi}_{\mathcal{A}_{A_{n-1}}}\big(\mathrm{p}(i),t\big)$, $\bar{\chi}_{\mathcal{A}_{B_n}}\big(\mathrm{p}(i),t\big)$, and $\bar{\chi}_{\mathcal{A}_{D_n}}\big(\mathrm{p}(i),t\big)$ are respectively deduced from \cite[Theorem~4.1]{Ar}, \cite[Theorem~4.2]{Ar}, and \cite[Theorem~4.3]{Ar}.
\end{proof}

\begin{example}
Using \texttt{SageMath}, we compute the coboundary polynomial $\bar{\chi}_{\mathcal{A}_{A_{12}}}(q,t)$, and deduce the Tutte polynomial $T_{\mathcal{A}_{12}}(x,y) = y^{66} + 12y^{65} + 78y^{64} + 364y^{63} + 1365y^{62} + 4368y^{61} + 12376y^{60} + 31824y^{59} + 75582y^{58} + 167960y^{57} + 13xy^{55} + 352716y^{56} + 143xy^{54} + 705419y^{55} + 858xy^{53} + 1351922y^{54} + 3718xy^{52} + 2495130y^{53} + 13013xy^{51} + 4452668y^{52} + 39039xy^{50} + 7708415y^{51} + 104104xy^{49} + 12981111y^{50} + 252824xy^{48} + 21313292y^{49} + 568854xy^{47} + 34183578y^{48} + 78x^2 y^{45} + 1200914xy^{46} + 53644734y^{47} + 780x^2 y^{44} + 2401750xy^{45} + 82488835y^{46} + 4290x^2 y^{43} + 4584372xy^{44} + 124439172y^{45} + 17160x^2 y^{42} + 8400392xy^{43} + 184366182y^{44} + 55770x^2y^{41} + 14845402xy^{42} + 268521682y^{43} + 156156x^2 y^{40} + 25395942xy^{41} + 384782112y^{42} + 390390x^2 y^{39} + 42181568xy^{40} + 542887488y^{41} + 892320x^2 y^{38} + 68193697xy^{39} + 754658542y^{40} + 286x^3 y^{36} + 1896180x^2 y^{37} + 107530137xy^{38} + 1034170357y^{39} + 2574x^3 y^{35} + 3792360x^2 y^{36} + 165670648xy^{37} + 1397857032y^{38} + 12870x^3 y^{34} + 7204626x^2 y^{35} + 249773810xy^{36} + 1864518656y^{37} + 47190x^3 y^{33} + 13092300x^2 y^{34} + 368979039xy^{35} + 2455199604y^{36} + 141570x^3 y^{32} + 22879350x^2 y^{33} + 534690507xy^{34} + 3192906717y^{35} + 368082x^3 y^{31} + 38610000x^2 y^{32} + 760812702xy^{33} + 4102137897y^{34} + 858858x^3 y^{30} + 63127350x^2 y^{31} + 1063901124xy^{32} + 5208196422y^{33} + 715x^4 y^{28} + 1840410x^3 y^{29} + 100267596x^2 y^{30} + 1463187583xy^{31} + 6536274030y^{32} + 5720x^4 y^{27} + 3682250x^3 y^{28} + 155060620x^2 y^{29} + 1980439175xy^{30} + 8110295765y^{31} + 25740x^4 y^{26} + 6961240x^3 y^{27} + 233921545x^2 y^{28} + 2639608972xy^{29} + 9951530017y^{30} + 85800x^4 y^{25} + 12540528x^3 y^{26} + 344807320x^2 y^{27} + 3466233342xy^{28} + 12076978342y^{29} + 235950x^4 y^{24} + 21659352x^3 y^{25} + 497305380x^2 y^{26} + 4486536197xy^{27} + 14497575092y^{28} + 566280x^4 y^{23} + 36027420x^3 y^{24} + 702617916x^2 y^{25} + 5726217783xy^{26} + 17216245827y^{27} + 1287x^5 y^{21} + 1226940x^4 y^{22} + 57915000x^3 y^{23} + 973411725x^2 y^{24} + 7208935734xy^{25} + 20225890047y^{26} + 9009x^5 y^{20} + 2460315x^4 y^{21} + 90240150x^3 y^{22} + 1323519990x^2 y^{23} + 8954517000xy^{24} + 23507359590y^{25} + 36036x^5 y^{19} + 4639635x^4 y^{20} + 136645509x^3 y^{21} + 1767482925x^2 y^{22} + 10976952480xy^{23} + 27027494637y^{24} + 108108x^5 y^{18} + 8308300x^4 y^{19} + 201522321x^3 y^{20} + 2319877131x^2 y^{21} + 13282216168xy^{22} + 30737260554y^{23} + 270270x^5 y^{17} + 14214200x^4y^{18} + 289949660x^3y^{19} + 2994411277x^2 y^{20} + 15865968404xy^{21} + 34570017053y^{22} + 1716x^6 y^{15} + 594594x^5 y^{16} + 23333310x^4 y^{17} + 407552860x^3 y^{18} + 3802864780x^2 y^{19} + 18711238546x y^{20} + 38439949834y^{21} + 10296x^6 y^{14} + 1204632x^5 y^{15} + 36911160x^4 y^{16} + 560321190x^3 y^{17} + 4754000680x^2 y^{18} + 21786110775xy^{19} + 42240691636y^{20} + 36036x^6 y^{13} + 2290860x^5 y^{14} + 56525040x^4 y^{15} + 754462566x^3 y^{16} + 5852533830x^2 y^{17} + 25041209765xy^{18} + 45844190249y^{19} + 96096x^6 y^{12} + 4117113x^5 y^{13} + 84041100x^4 y^{14} + 996144864x^3 y^{15} + 7097905815x^2 y^{16} + 28406477814x y^{17} + 49100042277 y^{18} + 1716x^7 y^{10} + 216216x^6 y^{11} + 7018011x^5 y^{12} + 121546425x^4 y^{13} + 1291054050x^3 y^{14} + 8482363890x^2 y^{15} + 31786670916xy^{16} + 51835895826y^{17} + 8580x^7y^9 + 456456x^6y^{10} + 11447436x^5y^{11} + 171372201x^4y^{12} + 1644037395x^3y^{13} + 9988126005x^2y^{14} + 35055564258xy^{15} + 53860118169y^{16} + 25740x^7y^8 + 900900x^6y^9 + 18078060x^5y^{10} + 236384148x^4y^{11} + 2058775719x^3y^{12} + 11583115401x^2y^{13} + 38049502062xy^{14} + 54968456400y^{15} + 1287x^8y^6 + 60060x^7y^7 + 1647360x^6y^8 + 27717690x^5y^9 + 319875556x^4y^{10} + 2536227408x^3y^{11} + 13213854732x^2y^{12} + 40561541020xy^{13} + 54956679987y^{14} + 5148x^8y^5 + 145860x^7y^6 + 2869152x^6y^7 + 41342730x^5y^8 + 424848710x^4y^9 + 3070406768x^3y^{10} + 14794737022x^2y^{11} + 42339276044xy^{12} + 53641257890y^{13} + x^{12} + 715x^9y^3 + 12870x^8y^4 + 308880x^7y^5 + 4975971x^6y^6 + 60537048x^5y^7 + 553491510x^4y^8 + 3642506725x^3y^9 + 16196480443x^2y^{10} + 43092548434xy^{11} + 50889618637y^{12} + 66x^{11} + 286x^{10}y + 2145x^9y^2 + 45045x^8y^3 + 626340x^7y^4 + 8270262x^6y^5 + 87038952x^5y^6 + 705403842x^4y^7 + 4213159665x^3y^8 + 17238306085x^2y^9 + 42518818056xy^{10} + 46659503176y^{11} + 1925x^{10} + 14300x^9y + 113685x^8y^2 + 1317030x^7y^3 + 13498485x^6y^4 + 121050930x^5y^5 + 870740871x^4y^6 + 4709432013x^3y^7 + 17688057387x^2y^8 + 40352717405xy^9 + 41042986985y^10 + 32670x^9 + 308022x^8y + 2449590x^7y^2 + 21010990x^6y^3 + 160690530x^5y^4 + 1023013992x^4y^5 + 5013315307x^3y^6 + 17283458049x^2y^7 + 36445235541xy^8 + 34305690705y^9 + 357423x^8 + 3740880x^7y + 28330302x^6y^2 + 193708515x^5y^3 + 1110769660x^4y^4 + 4969514550x^3y^5 + 15799102176x^2y^6 + 30871572303xy^7 + 26905602867y^8 + 2637558x^7 + 28139826x^6y + 193426233x^5y^2 + 1055099045x^4y^3 + 4427328620x^3y^4 + 13168529088x^2y^5 + 24040142841xy^6 + 19469297133y^7 + 13339535x^6 + 135486780x^5y + 801375861x^4y^2 + 3353221300x^3y^3 + 9638525325x^2y^4 + 16736474040xy^5 + 12703397135y^6 + 45995730x^5 + 416102258x^4y + 1979224104x^3y^2 + 5859961536x^2y^3 + 10012432430xy^4 + 7237654710y^5 + 105258076x^4 + 784515160x^3y + 2716328472x^2y^2 + 4857283288xy^3 + 3436086940y^4 + 150917976x^3 + 829158408x^2y + 1746026568xy^2 + 1263741336y^3 + 120543840x^2 + 397126080xy + 316499040y^2 + 39916800x + 39916800y$.
\end{example}

\section{Shifted Young Diagram} \label{SeYo}

\noindent We introduce some definitions on the shifted Young diagram, and associate a topology on it. Then, we expose the shifted Young diagrams associated to the positive root systems of the classical Lie algebras. These diagrams play a central role in our computing as \emph{the open sets of these diagrams correspond to the ideal arrangements \cite[Theorem 3.1]{CePa}.} 

\begin{definition}
A shifted Young diagram $Y$ is a finite collection of boxes arranged rows, $\mathbf{b}(i,j)$ designating the box of the $i^{\text{th}}$ row and the $j^{\text{th}}$ column, such that,
\begin{itemize}
\item if $Y$ has more than one row,
\item if $\mathbf{b}(i,l_i)$ resp. $\mathbf{b}(i,r_i)$ designates the leftmost resp. rightmost box on the $i^{\text{th}}$ row of $Y$,
\end{itemize}
then $l_i \leq l_{i+1}$ and $r_i \geq r_{i+1}$.
\end{definition}

\noindent If a shifted Young diagram $Y$ has $k$ rows, then the $k$-tuple $(r_1 -l_1 +1, \dots,r_k -l_k +1)$ is called its shape. The box set of $Y$ is $$\mathbf{B} := \big\{\mathbf{b}(i,j)\ |\ i \in [k],\, j \in \{l_k, \dots, r_k\}\big\}.$$

\begin{definition}
Take a shifted Young diagram $Y$, and a box $\mathbf{b}(i,j)$ of $\mathbf{B}$. The box set $\mathbf{B}_{i,j}$ of $Y$ generated by $\mathbf{b}(i,j)$ is
$$\mathbf{B}_{i,j} := \big\{\mathbf{b}(u,v) \in \mathbf{B}\ |\ u \geq i,\, v \geq j\big\}.$$
\end{definition}

\begin{lemma}
Define the set $\mathscr{O} := \big\{\mathbf{B}_{i,j}\ |\ i \in [k],\, j \in \{l_k, \dots, r_k\}\big\}$. Then, $\mathscr{O}$ is a basis for the topological space $(\mathbf{B}, \mathscr{O})$. We denote by $\mathscr{T}$ the topology of $(\mathbf{B}, \mathscr{O})$.
\end{lemma}

\begin{proof}
Let $\mathbf{B}_{i_1, j_1}, \mathbf{B}_{i_2, j_2} \in \mathscr{O}$, and $i = \max(i_1,i_2)$, $j = \max(j_1,j_2)$. Then $\mathbf{B}_{i_1, j_1} \cap \mathbf{B}_{i_2, j_2} = \mathbf{B}_{i,j}$ which belongs to $\mathscr{O}$.
\end{proof}

\begin{definition}
Let $O$ be an element of the topology $\mathscr{T}$ of a shifted Young diagram with $k$ rows. There exist an integer $m \leq k$, and $m$ tuples $(i_1,j_1), \dots, (i_m,j_m)$ such that
\begin{align*}
& (1)\ O = \bigcup_{l \in [m]} \mathbf{B}_{i_l, j_l},\\
& (2)\ \forall h,l \in [m],\, \mathbf{B}_{i_h, j_h} \nsubseteq \mathbf{B}_{i_l, j_l},\\
& (3)\ \forall h,l \in [m],\, h<l \Rightarrow i_h < i_l. 
\end{align*}
The set of generating boxes of $O$ is $G_O := \big(\mathbf{b}(i_l,j_l)\big)_{l \in [m]}$.
\end{definition}

\begin{definition}
We say that an open set $O$ of a shifted Young diagram $Y$ with $k$ rows is full if $\big\{\mathbf{b}(i,r_i)\ |\ i \in [k]\big\} \subseteq O$.
\end{definition}

\begin{example}
For the same shifted Young diagram, in Figure \ref{Young},
\begin{itemize}
\item the open set $O_1 = \big\{\mathbf{b}(3,4), \mathbf{b}(3,5), \mathbf{b}(3,6), \mathbf{b}(4,4), \mathbf{b}(4,5), \mathbf{b}(4,6), \mathbf{b}(5,4)\big\}$ is not full,
\item but is the open set $O_2 = \big\{\mathbf{b}(1,6), \mathbf{b}(1,7), \mathbf{b}(2,6), \mathbf{b}(2,7), \mathbf{b}(3,6), \mathbf{b}(4,4), \mathbf{b}(4,5), \mathbf{b}(4,6), \mathbf{b}(5,4)\big\}$.
\end{itemize}

\begin{figure}[h]
\centering
\includegraphics[scale=0.8]{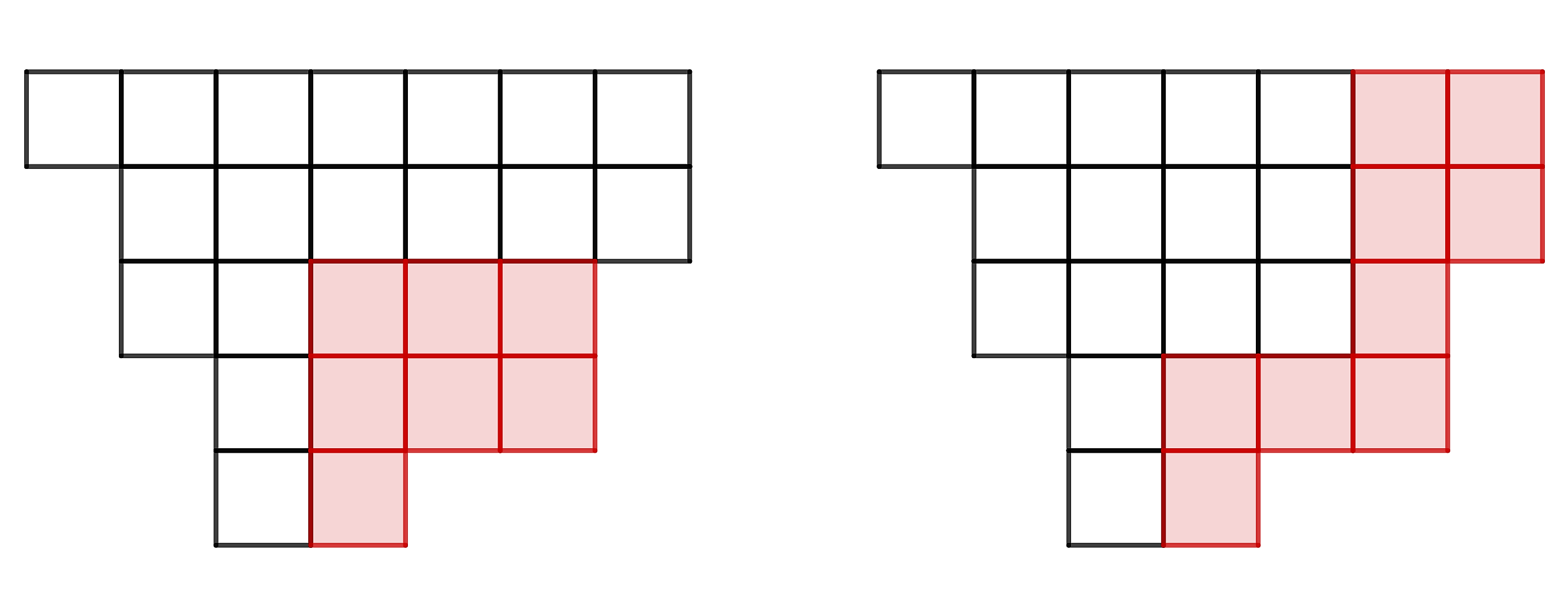}
\caption{The open Sets $O_1$ and $O_2$}
\label{Young}
\end{figure}
\end{example}

\noindent The next step helps us to see which hyperplanes play a role in our calculations. We adopt the following notation, with $i < j$:
\begin{itemize}
\item the tuple $(i,j)$ represents the hyperplane $\{x_i = x_j\}$,
\item the tuple $(i,-j)$ represents the hyperplane $\{x_i = -x_j\}$,
\item the tuple $(i,0)$ represents the hyperplane $\{x_i = 0\}$.
\end{itemize}

\noindent (A) For type $A_{n-1}$, we use the simple system $\Delta_{A_{n-1}} = \big\{\alpha_i = e_i - e_{i+1}\ |\ i \in [n-1]\big\}$. Its suitable Young diagram $Y_{A_{n-1}}$ has shape $(n-1, n-2, \dots, 1)$, and boxes $\mathbf{b}_A(i,j)$ filled with hyperplanes according to the assignment $\mathbf{b}_A(i,j) = (\alpha_i + \dots + \alpha_{n-j})^{\perp}$, $1 \leq i \leq j \leq n-1$. With the adopted notation, $\mathbf{b}_A(i,j) = (i,n-j+1)$, $1 \leq i \leq j \leq n-1$.

\begin{example}
The shifted Young diagram $Y_{A_7}$ is
$$\begin{array}{lllllll}
(1,8) & (1,7) & (1,6) & (1,5) & (1,4) & (1,3) & (1,2) \\
(2,8) & (2,7) & (2,6) & (2,5) & (2,4) & (2,3) &  \\
(3,8) & (3,7) & (3,6) & (3,5) & (3,4) &  &  \\
(4,8) & (4,7) & (4,6) & (4,5) &  &  &  \\
(5,8) & (5,7) & (5,6) &  &  &  &  \\
(6,8) & (6,7) &  &  &  &  &  \\
(7,8) &  &  &  &  &  &  \end{array}$$ 
\end{example}

\noindent (B) For type $B_n$, we use the simple system $\Delta_{B_n} = \big\{\alpha_i = e_i - e_{i+1}\ |\ i \in [n-1]\big\} \sqcup \{\alpha_n = e_n\}$. Its suitable Young diagram $Y_{B_n}$ has shape $(2n-1, 2n-3, \dots, 1)$, and boxes $\mathbf{b}_B(i,j)$ filled with hyperplanes according to the assignment
$$\mathbf{b}_B(i,j) = \left\{ \begin{array}{ll} \big( \alpha_i+\dots +\alpha_j+ 2(\alpha_{j+1}+\dots+\alpha_n) \big)^{\perp}, & 1 \leq j \leq n-1 \\
(\alpha_i+\dots +\alpha_{2n-j})^{\perp}, & n \leq j \leq 2n-1 \end{array} \right..$$
With the adopted notation, we have
$$\mathbf{b}_B(i,j) = \left\{ \begin{array}{ll} \big(i,-(j+1)\big), &  1 \leq i \leq j \leq n-1 \\
(i,0), &  j = n \\ (i, 2n-j+1), & n < j \leq 2n-1 \end{array} \right..$$

\begin{example}
The shifted Young diagram $Y_{B_6}$ is
$$\begin{array}{lllllllllll}
(1,-2) & (1,-3) & (1,-4) & (1,-5) & (1,-6) & (1,0) & (1,6) & (1,5) & (1,4) & (1,3) & (1,2) \\
       & (2,-3) & (2,-4) & (2,-5) & (2,-6) & (2,0) & (2,6) & (2,5) & (2,4) & (2,3) &  \\
             &  & (3,-4) & (3,-5) & (3,-6) & (3,0) & (3,6) & (3,5) & (3,4) &  &  \\
                   &  &  & (4,-5) & (4,-6) & (4,0) & (4,6) & (4,5)  &  &  &  \\
                         &  &  &  & (5,-6) & (5,0) & (5,6) &  &  &  &  \\
                               &  &  &  &  & (6,0) &  &  &  &  & \end{array}$$
\end{example}

\noindent (C) For type $C_n$, we use the simple system $\Delta_{C_n} = \big\{\alpha_i = e_i - e_{i+1}\ |\ i \in [n-1]\big\} \sqcup \{\alpha_n = 2e_n\}$. Its suitable Young diagram $Y_{C_n}$ has shape $(2n-1, 2n-3, \dots, 1)$, and boxes $\mathbf{b}_C(i,j)$ filled with hyperplanes according to the assignment
$$\mathbf{b}_C(i,j) = \left\{ \begin{array}{ll} \alpha_i+\dots +\alpha_{j-1}+ 2(\alpha_j+\dots+\alpha_{n-1}) + \alpha_n, & 1 \leq j \leq n-1 \\
\alpha_i+\dots +\alpha_{2n-j}, & n \leq j \leq 2n-1 \end{array} \right..$$
With the adopted notation, we have
$$\mathbf{b}_C(i,j) = \left\{ \begin{array}{ll} (i,0), & 1 \leq i = j \leq n \\ 
(i,-j), & 1 \leq i < j \leq n \\ (i, 2n-j+1), & n < j \leq 2n-1 \end{array} \right..$$

\begin{example}
The shifted Young diagram $Y_{C_6}$ is
$$\begin{array}{lllllllllll}
(1,0) & (1,-2) & (1,-3) & (1,-4) & (1,-5) & (1,-6) & (1,6) & (1,5) & (1,4) & (1,3) & (1,2) \\
       & (2,0) & (2,-3) & (2,-4) & (2,-5) & (2,-6) & (2,6) & (2,5) & (2,4) & (2,3) &  \\
             &  & (3,0) & (3,-4) & (3,-5) & (3,-6) & (3,6) & (3,5) & (3,4) &  &  \\
                   &  &  & (4,0) & (4,-5) & (4,-6) & (4,6) & (4,5)  &  &  &  \\
                         &  &  &  & (5,0) & (5,-6) & (5,6) &  &  &  &  \\
                               &  &  &  &  & (6,0) &  &  &  &  & \end{array}$$
\end{example}

\noindent (D) For type $D_n$, we use the simple system $\Delta_{D_n} = \big\{\alpha_i = e_i - e_{i+1}\ |\ i \in [n-1]\big\} \sqcup \{\alpha_n = e_{n-1} + e_n\}$. Its suitable Young diagram $Y_{D_n}$ has shape $(2n-2, 2n-4, \dots, 2)$, and boxes $\mathbf{b}_D(i,j)$ filled with hyperplanes according to the assignment
$$\mathbf{b}_D(i,j) = \left\{ \begin{array}{ll} \big( \alpha_i+\dots +\alpha_j+ 2(\alpha_{j+1}+\dots+\alpha_{n-2}) + \alpha_{n-1} + \alpha_n \big)^{\perp}, & 1 \leq j \leq n-2 \\
(\alpha_i+\dots +\alpha_{n-2} + \alpha_n)^{\perp}, & j = n-1 \\
(\alpha_i+\dots +\alpha_{2n-j})^{\perp}, & n \leq j \leq 2n-1 \end{array} \right..$$
With the adopted notation, we have
$$\mathbf{b}_D(i,j) = \left\{ \begin{array}{ll} \big(i,-(j+1)\big), & 1 \leq i \leq j \leq n-1 \\
(i, 2n-j), & n-1 < j \leq 2n-2 \end{array} \right..$$

\begin{example}
The shifted Young diagram $Y_{D_6}$ is
$$\begin{array}{llllllllll}
(1,-2) & (1,-3) & (1,-4) & (1,-5) & (1,-6) & (1,6) & (1,5) & (1,4) & (1,3) & (1,2) \\
       & (2,-3) & (2,-4) & (2,-5) & (2,-6) & (2,6) & (2,5) & (2,4) & (2,3) &  \\
             &  & (3,-4) & (3,-5) & (3,-6) & (3,6) & (3,5) & (3,4) &  &  \\
                   &  &  & (4,-5) & (4,-6) & (4,6) & (4,5)  &  &  &  \\
                         &  &  &  & (5,-6) & (5,6) &  &  &  &  \end{array}$$
\end{example}

\section{Partition in Accordance} \label{SePa}

\noindent The partition in accordance with an ideal complement is a partition from which we compute the coboundary polynomial of this ideal. Each box of a shifted Young diagram we consider represents now a hyperplane $(i,j)$. If $(i,j)$ corresponds to the box $\mathbf{b}(u,v)$, we use $(i,j)$ instead of $\mathbf{b}(u,v)$, and write $\mathbf{B}(i,j)$ for the box set $\mathbf{B}_{u,v}$ to emphasize the hyperplane context. Moreover, since a box set is a hyperplane arrangement, we can consider an ideal arrangement as a subset of $Y_{X_m}$, where $X_m$ is any type in $\{A_{n-1}, B_n, C_n, D_n\}$. Remark that, since the coboundary polynomial of an open set is equal the product of the coboundary polynomials associated to its connected components, we just need to consider the full connected open sets or, equivalently, the full connected ideal arrangements.

\begin{definition}
Take a full connected ideal arrangement $\mathcal{A}_I$ of $\Phi_{X_m}$, with generating boxes $G_I = \big((i_l,j_l)\big)_{l \in [k]}$. The signature $s_I(i)$ of an integer $i$ in accordance with $I$ is 
$$s_I(i) := \big\{l \in [k]\ |\ \text{$i$ appears in at least one box of $\mathbf{B}(i_l,j_l)$}\big\}.$$
\end{definition}

\noindent Define the integer set of $I$ by $\mathbb{I}_I := \{i \in \mathbb{Z}\ |\ \text{$i$ appears in at least one box of $\mathcal{A}_I$}\}$.

\paragraph*{Algorithm P.} We transform $[m]$ into the partition $A^{(1)}| \dots |A^{(r)}|B^{(1)}| \dots |B^{(s)}$ in accordance with an ideal $I$.\\
\textbf{P-1.} We partition $[m]$ in $A|B$ such that $$\forall i \in A:\, -i \notin \mathbb{I}_I \quad \text{and} \quad \forall i \in B:\, -i \in \mathbb{I}_I.$$
\textbf{P-2.} We partition $A$ in $A^{(1)}| \dots |A^{(r)}$ such that
\begin{itemize}
\item $\forall i \in [r],\, \forall u,v \in A^{(i)},\, s_I(u) = s_I(v)$,
\item $\forall i, j \in [r],\, \forall u \in A^{(i)},\, \forall v \in A^{(j)},\, i \neq j \Rightarrow s_I(u) \neq s_I(v)$.
\end{itemize}
\textbf{P-3.} We partition $B$ in $B^{(1)}| \dots |B^{(s)}$ such that
\begin{itemize}
\item $\forall i \in [s],\, \forall u,v \in B^{(i)},\, s_I(-u) = s_I(-v)$,
\item $\forall i, j \in [s],\, \forall u \in B^{(i)},\, \forall v \in B^{(j)},\, i \neq j \Rightarrow s_I(-u) \neq s_I(-v)$.
\end{itemize}

\begin{definition}
Let $\mathcal{A}_I$ be a full connected ideal arrangement of $\Phi_{X_m}$. With the notation of Algorithm P, the partition $P_I$ of $[m]$ in accordance with $I$ is $P_I := A^{(1)}| \dots |A^{(r)}|B^{(1)}| \dots |B^{(s)}$. 
\end{definition}

\noindent (A) Let $\mathbf{B}_{A_{n-1}}$ be the box set of the diagram $Y_{A_{n-1}}$. The box set generated by a hyperplane $(i,j)$ of $\mathbf{B}_{A_{n-1}}$ is
$$\mathbf{B}_{A_{n-1}}(i,j) = \big\{(u,v)\ |\ i \leq u < v \leq j\big\}.$$
Let $\mathcal{A}_I$ be a full connected ideal arrangement of $\Phi_{A_{n-1}}$, $G_I = \big((i_l,j_l)\big)_{l \in [k]}$, and $u$ a nonnegative integer. Then,
$$s_I(u) = \{l \in [k]\ |\ i_l \leq u \leq j_l\}.$$

\begin{example} \label{ExA}
The ideal arrangement $\mathcal{A}_{I_a}$ of $\Phi_{A_7}$, with $G_{I_a} = \big((1,3),(2,5),(4,7),(6,8)\big)$, is    
$$\mathcal{A}_{I_a} = \begin{array}{lllllll}
 &  &  &  &  & (1,3) & (1,2) \\
 &  &  & (2,5) & (2,4) & (2,3) &  \\
 &  &  & (3,5) & (3,4) &  &  \\
 & (4,7) & (4,6) & (4,5) &  &  &  \\
 & (5,7) & (5,6) &  &  &  &  \\
(6,8) & (6,7) &  &  &  &  &  \\
(7,8) &  &  &  &  &  &  \end{array}$$

\noindent The signatures in accordance with $I_a$ are

\begin{center}
    \begin{tabular}{c|cccccccc}
$\mathbf{i}$ & $1$ & $2$ & $3$ & $4$ & $5$ & $6$ & $7$ & $8$ \\ \hline    
$\mathbf{s_{I_a}(i)}$ & $\{1\}$ & $\{1,2\}$ & $\{1,2\}$ & $\{2,3\}$ & $\{2,3\}$ & $\{3,4\}$ & $\{3,4\}$ & $\{4\}$     
    \end{tabular}
\end{center} 

\noindent The partition of $[8]$ in accordance with $I_a$ is $P_{I_a} = \{1\}|\{2,3\}|\{4,5\}|\{6,7\}|\{8\}$.
\end{example}

\begin{lemma} \label{LeSA}
Let $\mathcal{A}_I$ be a full connected ideal arrangement of $\Phi_{A_{n-1}}$ with associated partition $A^{(1)}| \dots |A^{(r)}$. Then,
\begin{align*}
& \mathcal{A}_I = \bigsqcup_{i=1}^r \binom{A^{(i)}}{2} \, \sqcup \, \big( A^{(i)} \times \bigsqcup_{j \in R^{(i)}} A^{(j)} \big) \\
& \text{with} \quad R^{(i)} = \big\{j \in \{i+1, \dots, r\}\ |\ s_I(A^{(i)}) \, \cap \, s_I(A^{(j)}) \neq \emptyset \big\}.
\end{align*}
\end{lemma}
 
\begin{proof}
Let $G_I = \big((i_l,j_l)\big)_{l \in [k]}$:
\begin{itemize}
\item If $l \in s_I(A^{(i)})$, then $\binom{A^{(i)}}{2} \subseteq \mathbf{B}_{A_{n-1}}(i_l,j_l)$.
\item And if $l \in s_I(A^{(i)}) \cap s_I(A^{(j)})$, then $A^{(i)} \times A^{(j)} \subseteq \mathbf{B}_{A_{n-1}}(i_l,j_l)$.
\end{itemize}
Therefore, $$\bigsqcup_{i=1}^r \binom{A^{(i)}}{2} \, \sqcup \, \big( A^{(i)} \times \bigsqcup_{j \in R^{(i)}} A^{(j)} \big) \subseteq \mathcal{A}_I.$$
Now, take $(x,y) \in \mathcal{A}_I$, which means $(x,y) \in \mathbf{B}_{A_{n-1}}(i_l,j_l)$ for $l \in [k]$. There are $i,j \in [r]$ such that $x \in A^{(i)}$, and $y \in A^{(j)}$:
\begin{itemize}
\item If $i=j$, then $(x,y) \in \binom{A^{(i)}}{2}$.
\item If $i \neq j$, since $s_I(A^{(i)}) \cap s_I(A^{(j)}) \neq \emptyset$, then $(x,y) \in A^{(i)} \times A^{(j)}$.
\end{itemize}
\end{proof}

\noindent (B) Define the linear order $\prec_b$ on $\{-n, \dots, -1, 0, 1, \dots, n\}$ by
$$1 \prec_b 2 \prec_b \dots n \prec_b 0 \prec_b -n \prec_b \dots \prec_b -2 \prec_b -1.$$
Consider a box $(i,j)$ of $Y_{B_n}$. The box subset $\mathbf{B}_{B_n}(i,j)$ generated by $(i,j)$ is
$$\mathbf{B}_{B_n}(i,j) = \big\{(u,v)\ |\ i \preceq_b u \preceq_b n,\, u \prec_b v \preceq_b j\big\}.$$
Let $\mathcal{A}_I$ be a full connected ideal arrangement of $\Phi_{B_n}$, $G_I = \big((i_l,j_l)\big)_{l \in [k]}$, and $u$ a nonnegative integer. Then, $$s_I(u) = \{l \in [k]\ |\ i_l \preceq_b u \preceq_b j_l\}.$$

\begin{example} \label{ExB}
The ideal arrangement $\mathcal{A}_{I_b}$ of $\Phi_{B_6}$, such that $G_{I_b} = \big\{(1,4),(2,0),(4,-5)\big\}$, is
$$\mathcal{A}_{I_b} = \begin{array}{llllllll}
                              &  &  &  &  & (1,4) & (1,3) & (1,2) \\
               &  & (2,0) & (2,6) & (2,5) & (2,4) & (2,3) &  \\
               &  & (3,0) & (3,6) & (3,5) & (3,4) &  &  \\
  (4,-5) & (4,-6) & (4,0) & (4,6) & (4,5)  &  &  &  \\
         & (5,-6) & (5,0) & (5,6) &  &  &  &  \\
               &  & (6,0) &  &  &  &  &  \end{array}$$

\noindent The signatures in accordance with $I_b$ are

\begin{center}
    \begin{tabular}{c|ccccccccc}
$\mathbf{i}$ & $1$ & $2$ & $3$ & $4$ & $5$ & $6$ & $0$ & $-6$ & $-5$ \\ \hline   
$\mathbf{s_{I_b}(i)}$ & $\{1\}$ & $\{1,2\}$ & $\{1,2\}$ & $\{1,2,3\}$ & $\{2,3\}$ & $\{2,3\}$ & $\{2,3\}$ & $\{3\}$ & $\{3\}$
    \end{tabular}
    \end{center}

\noindent The partition of $[6]$ according to $I_b$ is $P_{I_b} = \{1\}|\{2,3\}|\{4\}|\{5,6\}$.
\end{example}

\noindent Let $B$ be a subset of $\mathbb{N}^*$. Define
$$\tilde{B} := \{\pm i\ |\ i \in B\},\ \text{and}\ \begin{Bmatrix} \tilde{B} \\ 2 \end{Bmatrix} := \big\{ (i, \pm j)\ |\ i,j \in B,\, i<j \big\}.$$

\begin{lemma} \label{LeB}
Take a full connected ideal arrangement $\mathcal{A}_I$ of $\Phi_{B_n}$ with associated partition $A^{(1)}| \dots |A^{(r)}|B^{(1)}| \dots |B^{(s)}$. Let,
\begin{align*}
\bullet \ \mathcal{A}_{I_A} = & \ \bigsqcup_{i=1}^r \binom{A^{(i)}}{2} \, \sqcup \, \big( A^{(i)} \times \bigsqcup_{l \in R^{(i)}} A^{(l)} \big) \\
& \qquad \text{with} \quad R^{(i)} = \big\{l \in \{i+1, \dots, r\}\ |\ s_I(A^{(i)}) \cap s_I(A^{(l)}) \neq \emptyset \big\},\\
\bullet \ \mathcal{A}_{I_B} = & \ \bigsqcup_{j=1}^s \begin{Bmatrix} \tilde{B}^{(j)} \\ 2 \end{Bmatrix} \, \sqcup \, \big( \bigsqcup_{l \in R_A^{(j)}} A^{(l)} \sqcup \bigsqcup_{h \in S^{(j)}} B^{(h)} \big) \times \tilde{B}^{(j)} \\
& \qquad \text{with} \quad R_A^{(j)} = \big\{l \in [r]\ |\ s_I(B^{(j)}) \cap s_I(A^{(l)}) \neq \emptyset \big\} \\
& \qquad \text{and} \quad S^{(j)} = \big\{h \in [j-1]\ |\ s_I(B^{(j)}) \cap s_I(B^{(h)}) \neq \emptyset \big\}, \\
\bullet \ \mathcal{A}_{I_0} = & \ \big( \bigsqcup_{l \in R_0} A^{(l)} \sqcup \bigsqcup_{h \in S_0} B^{(h)} \big) \times \{0\} \\
& \qquad \text{with} \quad R_0 = \big\{l \in [r]\ |\ s_I(0) \cap s_I(A^{(l)}) \neq \emptyset \big\} \\
& \qquad \text{and} \quad S_0 = \big\{h \in [s]\ |\ s_I(0) \cap s_I(B^{(h)}) \neq \emptyset \big\}.
\end{align*}
Then $\mathcal{A}_I = \mathcal{A}_{I_A} \sqcup \mathcal{A}_{I_B} \sqcup \mathcal{A}_{I_0}$.
\end{lemma}

\begin{proof}
Let $G_I = \big((i_l,j_l)\big)_{l \in [k]}$:
\begin{itemize}
\item If $l \in s_I(A^{(i)})\ \text{resp.}\ s_I(B^{(j)})$, then $\begin{pmatrix} A^{(i)} \\ 2 \end{pmatrix}\ \text{resp.}\  \begin{Bmatrix} \tilde{B}^{(j)} \\ 2 \end{Bmatrix} \subseteq \mathbf{B}_{B_n}(i_l,j_l)$.
\item If $l \in s_I(A^{(i)}) \cap s_I(A^{(l)})$, then $A^{(i)} \times A^{(l)} \subseteq \mathbf{B}_{B_n}(i_l,j_l)$.
\item If $l \in s_I(A^{(i)}) \cap s_I(B^{(j)})$, then $A^{(i)} \times \tilde{B}^{(j)} \subseteq \mathbf{B}_{B_n}(i_l,j_l)$.
\item If $l \in s_I(B^{(h)}) \cap s_I(B^{(j)})$, $i<j$,  then $B^{(h)} \times \tilde{B}^{(j)} \subseteq \mathbf{B}_{B_n}(i_l,j_l)$.
\item If $l \in s_I(A^{(i)}) \cap s_I(0)\ \text{resp.}\ s_I(B^{(j)}) \cap s_I(0)$, then $A^{(i)} \times \{0\}\ \text{resp.}\ B^{(j)} \times \{0\} \subseteq \mathbf{B}_{B_n}(i_l,j_l)$.
\end{itemize}
So, $\mathcal{A}_{I_A} \sqcup \mathcal{A}_{I_B} \sqcup \mathcal{A}_{I_0} \subseteq \mathcal{A}_I$.\\
Now, take $(x,y) \in \mathcal{A}_I$, which means $(x,y) \in \mathbf{B}_{B_n}(i_l,j_l)$ for $l \in [k]$. If $y \neq 0$, then there are $i,j \in [r] \cup [s]$ such that $x \in A^{(i)} \cup B^{(i)}$, and $y \in A^{(j)} \cup B^{(j)}$:
\begin{itemize}
\item If $i=j$, then $(x,y) \in \begin{pmatrix} A^{(i)} \\ 2 \end{pmatrix} \cup \begin{Bmatrix} \tilde{B}^{(i)} \\ 2 \end{Bmatrix}$.
\item If $i \neq j$, suppose for example that $x \in A^{(i)}$, and $y \in B^{(j)}$. Since $s_I(A^{(i)}) \cap s_I(B^{(j)}) \neq \emptyset$, then $(x,y) \in A^{(i)} \times \tilde{B}^{(j)}$. The proof is analogous for the other cases.
\end{itemize}
\end{proof}

\noindent (C) Define the linear order $\prec_c$ on $\{-n, \dots, -1, 1, \dots, n\}$ by
$$1 \prec_c 2 \prec_c \dots n \prec_c -n \prec_c \dots \prec_c -2 \prec_c -1.$$
Consider a box $(i,j)$ of $Y_{C_n}$. The box subset $\mathbf{B}_{C_n}(i,j)$ generated by $(i,j)$ is
$$\mathbf{B}_{C_n}(i,j) := \left\{ \begin{array}{ll} 
\big\{(u,v)\ |\ i \preceq_c u \prec_c v \preceq_c j\big\} & \text{if}\ j>0, \\
\big\{(u,v)\ |\ i \preceq_c u \preceq_c n-1,\, u \prec_c v \preceq_c j\big\} \sqcup \big\{(u,0)\ |\ -j \preceq_c u \preceq_c n\big\} & \text{if}\ j<0, \\
Y_{C_n}\big(i,-(i+1)\big) \sqcup \big\{(u,0)\ |\ i \preceq_c u \preceq_c n\big\} & \text{if}\ j=0.
\end{array} \right.$$
Let $\mathcal{A}_I$ be a full connected ideal arrangement of $\Phi_{C_n}$, $G_I = \big((i_l,j_l)\big)_{l \in [k]}$, and $u$ a nonnegative integer. Then,
$$s_I(x) = \left\{ \begin{array}{ll}
\{l \in [k]\ |\ i_l \preceq_c i \preceq_c j_l\} & \text{if}\ x \neq 0, \\
\{l \in [k]\ |\ j_l \leq 0\} & \text{if}\ x=0.
\end{array} \right.$$

\begin{example} \label{ExC}
The ideal arrangement $\mathcal{A}_{I_c}$ of $\Phi_{C_6}$, such that $G_{I_c} = \big\{(1,4),(2,-6),(4,0)\big\}$, is
$$\mathcal{A}_{I_c} = \begin{array}{lllllllllll}
                            &  &  &  &  & (1,4) & (1,3) & (1,2) \\
            &  & (2,-6) & (2,6) & (2,5) & (2,4) & (2,3) &  \\
            &  & (3,-6) & (3,6) & (3,5) & (3,4) &  &  \\
(4,0) & (4,-5) & (4,-6) & (4,6) & (4,5)  &  &  &  \\
       & (5,0) & (5,-6) & (5,6) &  &  &  &  \\
             &  & (6,0) &  &  &  &  &  \end{array}$$

\noindent The signatures in accordance with $I_c$ are

\begin{center}
    \begin{tabular}{c|ccccccccc}
$\mathbf{i}$ & $1$ & $2$ & $3$ & $4$ & $5$ & $6$ & $-6$ & $-5$ & $0$ \\ \hline   
$\mathbf{s_{I_c}(i)}$ & $\{1\}$ & $\{1,2\}$ & $\{1,2\}$ & $\{1,2,3\}$ & $\{2,3\}$ & $\{2,3\}$ & $\{2,3\}$ & $\{3\}$ & $\{2,3\}$ 
    \end{tabular}
    \end{center}

\noindent The partition of $[6]$ in accordance with $I_c$ is $P_{I_c} = \{1\}|\{2,3\}|\{4\}|\{5\}|\{6\}$.
\end{example}

\begin{lemma} \label{LeC}
Take a full connected ideal arrangement $\mathcal{A}_I$ of $\Phi_{C_n}$ with associated partition $A^{(1)}| \dots |A^{(r)}|B^{(1)}| \dots |B^{(s)}$. Let,
\begin{align*}
\bullet \ \mathcal{A}_{I_A} = & \ \bigsqcup_{i=1}^r \binom{A^{(i)}}{2} \, \sqcup \, \big( A^{(i)} \times \bigsqcup_{l \in R^{(i)}} A^{(l)} \big) \\
& \qquad \text{with} \quad R^{(i)} = \big\{l \in \{i+1, \dots, r\}\ |\ s_I(A^{(i)}) \cap s_I(A^{(l)}) \neq \emptyset \big\},\\
\bullet \ \mathcal{A}_{I_B} = & \ \bigsqcup_{j=1}^s \begin{Bmatrix} \tilde{B}^{(j)} \\ 2 \end{Bmatrix} \, \sqcup \, \big( \bigsqcup_{l \in R_A^{(j)}} A^{(l)} \sqcup \bigsqcup_{h \in S^{(j)}} B^{(h)} \big) \times \tilde{B}^{(j)} \\
& \qquad \text{with} \quad R_A^{(j)} = \big\{l \in [r]\ |\ s_I(B^{(j)}) \cap s_I(A^{(l)}) \neq \emptyset \big\} \\
& \qquad \text{and} \quad S^{(j)} = \big\{h \in [j-1]\ |\ s_I(B^{(j)}) \cap s_I(B^{(h)}) \neq \emptyset \big\}, \\
\bullet \ \mathcal{A}_{I_0} = & \ \big( \bigsqcup_{l \in R_0} A^{(l)} \sqcup \bigsqcup_{h \in S_0} B^{(h)} \big) \times \{0\} \\
& \qquad \text{with} \quad R_0 = \big\{l \in [r]\ |\ s_I(0) \cap s_I(A^{(l)}) \neq \emptyset \big\} \\
& \qquad \text{and} \quad S_0 = \big\{h \in [s]\ |\ s_I(0) \cap s_I(B^{(h)}) \neq \emptyset \big\}.
\end{align*}
Then $\mathcal{A}_I = \mathcal{A}_{I_A} \sqcup \mathcal{A}_{I_B} \sqcup \mathcal{A}_{I_0}$.
\end{lemma}

\begin{proof}
It is analogous to the proof of Lemma \ref{LeB}.
\end{proof}

\noindent (D) Consider a box $(i,j)$ of $Y_{D_n}$. The box subset $\mathbf{B}_{D_n}(i,j)$ generated by $(i,j)$ is
$$\mathbf{B}_{D_n}(i,j) = \big\{(u,v)\ |\ i \preceq_c u \preceq_c n-1,\, u \prec_c v \preceq_c j\big\}.$$
Let $\mathcal{A}_I$ be a full connected ideal arrangement of $\Phi_{D_n}$, $G_I = \big((i_l,j_l)\big)_{l \in [k]}$, and $u$ a nonnegative integer. Then,
$$s_I(u) = \{l \in [k]\ |\ i_l \preceq_c u \preceq_c j_l\}.$$

\begin{example} \label{ExD}
The ideal arrangement $\mathcal{A}_{I_d}$ of $\Phi_{D_6}$, such that $G_{I_d} = \big\{(1,3),(2,6),(4,-5)\big\}$, is
$$\mathcal{A}_{I_d} = \begin{array}{llllllllll}
                            &  &  &  &  & (1,3) & (1,2) \\
             &  & (2,6) & (2,5) & (2,4) & (2,3) &  \\
             &  & (3,6) & (3,5) & (3,4) &  &  \\
(4,-5) & (4,-6) & (4,6) & (4,5)  &  &  &  \\
       & (5,-6) & (5,6) &  &  &  &  \end{array}$$

\noindent The signatures in accordance with $\mathcal{A}_{I_d}$ are

\begin{center}
    \begin{tabular}{c|ccccccccc}
$\mathbf{i}$ & $1$ & $2$ & $3$ & $4$ & $5$ & $6$ & $-6$ & $-5$ \\ \hline   
$\mathbf{s_{I_d}(i)}$ & $\{1\}$ & $\{1,2\}$ & $\{1,2\}$ & $\{2,3\}$ & $\{2,3\}$ & $\{2,3\}$ & $\{3\}$ & $\{3\}$
    \end{tabular}
\end{center}

\noindent The partitions of $[6]$ in accordance with $I_d$ is $P_{I_d} = \{1\}|\{2,3\}|\{4\}|\{5,6\}$.
\end{example}

\begin{lemma} \label{LeD}
Take a full connected ideal arrangement $\mathcal{A}_I$ of $\Phi_{D_n}$, with associated partition $A^{(1)}| \dots |A^{(r)}|B^{(1)}| \dots |B^{(s)}$. Let,
\begin{align*}
\bullet \ \mathcal{A}_{I_A} = & \ \bigsqcup_{i=1}^r \binom{A^{(i)}}{2} \, \sqcup \, \big( A^{(i)} \times \bigsqcup_{l \in R^{(i)}} A^{(l)} \big) \\
& \qquad \text{with} \quad R^{(i)} = \big\{l \in \{i+1, \dots, r\}\ |\ s_I(A^{(i)}) \cap s_I(A^{(l)}) \neq \emptyset \big\},\\
\bullet \ \mathcal{A}_{I_B} = & \ \bigsqcup_{j=1}^s \begin{Bmatrix} \tilde{B}^{(j)} \\ 2 \end{Bmatrix} \, \sqcup \, \big( \bigsqcup_{l \in R_A^{(j)}} A^{(l)} \sqcup \bigsqcup_{h \in S^{(j)}} B^{(h)} \big) \times \tilde{B}^{(j)} \\
& \qquad \text{with} \quad R_A^{(j)} = \big\{l \in [r]\ |\ s_I(B^{(j)}) \cap s_I(A^{(l)}) \neq \emptyset \big\} \\
& \qquad \text{and} \quad S^{(j)} = \big\{h \in [j-1]\ |\ s_I(B^{(j)}) \cap s_I(B^{(h)}) \neq \emptyset \big\}.
\end{align*}
Then $\mathcal{A}_I = \mathcal{A}_{I_A} \sqcup \mathcal{A}_{I_B}$.
\end{lemma}

\begin{proof}
It is analogous to the proof of Lemma \ref{LeB}.
\end{proof}

\section{Hyperplane Counting} \label{SeCob}

\noindent We compute the coboundary polynomial $\bar{\chi}_{\mathcal{A}_I}\big(\mathrm{p}(i),t\big)$ of an ideal arrangement associated to a classical root system, for prime numbers $\mathrm{p}(i)$ strictly bigger than $2$. That computing is based on the finite field method, the minors associated to the classical root systems, and the partition in accordance with an ideal. By means of the polynomial interpolation of Lagrange, one can deduce the deduce the coboundary polynomial
$$\bar{\chi}_{\mathcal{A}_I}(q,t) = \sum_{k=0}^{\#\mathcal{A}_I} \sum_{i \in \big[\mathrm{r}(\mathcal{A}_I) +1\big]} \Big(\prod_{\substack{j \in \big[\mathrm{r}(\mathcal{A}_I)+1\big] \\ j \neq i}} \frac{q - \mathrm{p}(j)}{\mathrm{p}(i) - \mathrm{p}(j)}\Big) \times [t^k]\bar{\chi}_{\mathcal{A}_I}\big(\mathrm{p}(i),t\big) \times t^k.$$
We keep the notation of Section~\ref{SeYo} stating that a tuple $(i,j)$ represents a hyperplane. The examples in this section are computed with \texttt{SageMath}.

\begin{lemma} \label{LePr}
Take two subsets $A,B$ of $[n]$ such that, for every $i\in A$, and $j \in B$, we have $i<j$. For $\bar{x} = (\bar{x}_1, \dots, \bar{x}_n) \in \mathbb{F}_q^n$, define the sets
$$a_i(\bar{x}) := \{u \in A\ |\ \bar{x}_u = \bar{i}\} \quad \text{and} \quad b_i(\bar{x}) := \{u \in B\ |\ \bar{x}_u = \bar{i}\}.$$
\begin{itemize}
\item[(1)] Consider the hyperplane arrangement $\mathcal{A} = \begin{pmatrix} A \\ 2 \end{pmatrix}$. Then
$$\#\bar{\mathcal{A}}(\bar{x})\ =\ \sum_{i=0}^{q-1} \begin{pmatrix} \#a_i(\bar{x}) \\ 2 \end{pmatrix}.$$
\item[(2)] Consider the hyperplane arrangement $\mathcal{A} = A \times B$. Then
$$\#\bar{\mathcal{A}}(\bar{x})\ =\ \sum_{i=0}^{q-1} \#a_i(\bar{x}) \times \#b_i(\bar{x}).$$
\item[(3)] Consider the hyperplane arrangement $\mathcal{A} = \begin{Bmatrix} \tilde{A} \\ 2 \end{Bmatrix}$. Then $$\#\bar{\mathcal{A}}(\bar{x})\ =\ \#a_0(\bar{x})^2 - \#a_0(\bar{x})\, +\, \sum_{i=1}^{q-1} \begin{pmatrix} \#a_i(\bar{x}) \\ 2 \end{pmatrix}\, +\, \sum_{j=1}^{\frac{q-1}{2}} \#a_j(\bar{x}) \times \#a_{-j}(\bar{x}).$$
\item[(4)] Consider the hyperplane arrangement $\mathcal{A} = A \times \tilde{B}$. Then
$$\#\bar{\mathcal{A}}(\bar{x})\ =\ 2 \times \#a_0(\bar{x}) \times \#b_0(\bar{x})\, +\, \sum_{i=1}^{q-1} \#a_i(\bar{x}) \times \#b_i(\bar{x})\, +\, \sum_{j=1}^{q-1} \#a_j(\bar{x}) \times \#b_{-j}(\bar{x}).$$
\item[(5)] Consider the hyperplane arrangement $\mathcal{A} = A \times \{0\}$. Then
$$\#\bar{\mathcal{A}}(\bar{x})\ =\ \#a_0(\bar{x}).$$
\end{itemize}
\end{lemma}

\begin{proof}
(1) Let $i,j \in A$ with $i<j$. Then, $\bar{x} \in (i,j)$ if and only if $\bar{x}_i = \bar{x}_j$.\\
(2) Let $i \in A$, $j \in B$. Then, $\bar{x} \in (i,j)$ if and only if $\bar{x}_i = \bar{x}_j$.\\
(3) Let $i,j \in A$ with $i<j$. Then, $\bar{x} \in (i,-j)$ if and only if $\bar{x}_i = -\bar{x}_j$.\\
(4) Let $i \in A$, $j \in B$. Then, $\bar{x} \in (i,-j)$ if and only if $\bar{x}_i = -\bar{x}_j$.\\
(5) Let $i \in A$. Then, $\bar{x} \in (i,0)$ if and only if $\bar{x}_i = \bar{0}$.
\end{proof}

\begin{theorem} \label{PrA}
Let $\mathcal{A}_I$ be a full connected ideal arrangement of $\Phi_{A_{n-1}}$, with associated partition $A^{(1)}| \dots |A^{(r)}$, and let $R^{(u)} = \big\{v \in \{u+1, \dots, r\}\ |\ s_I(A^{(u)}) \cap s_I(A^{(v)}) \neq \emptyset \big\}$. Then, for a positive integer $i$, we have
$$\bar{\chi}_{\mathcal{A}_I}\big(\mathrm{p}(i),t\big) = \sum_{\substack{a_1^{(1)} + \dots + a_{\mathrm{p}(i)}^{(1)}\ =\ \# A^{(1)} \\ \vdots \\ a_1^{(r)} + \dots + a_{\mathrm{p}(i)}^{(r)}\ =\ \# A^{(r)}}} \prod_{u=1}^r \binom{\# A^{(u)}}{a_1^{(u)}, \dots, a_{\mathrm{p}(i)}^{(u)}} \frac{t^{\sum_{s=1}^{\mathrm{p}(i)} \binom{a_s^{(u)}}{2} + a_s^{(u)} \sum_{v \in R^{(u)}} a_s^{(v)}}}{\mathrm{p}(i)}.$$
\end{theorem}

\begin{proof}
For a vector $\bar{x} = (\bar{x}_1, \dots, \bar{x}_n)$ in $\mathbb{F}_p^n$, define the set $a_k^{(i)}(\bar{x}) := \{u \in A^{(i)}\ |\ \bar{x}_u = \bar{k}\}$. We have $\dim \cap \mathcal{A}_I = 1$ for every full connected ideal arrangement $\mathcal{A}_I$ of $\Phi_{A_{n-1}}$. Then,
\begin{align*}
\mathrm{p}(i) \bar{\chi}_{\mathcal{A}_I}(q,t) & = \sum_{\bar{x} \in \mathbb{F}_{\mathrm{p}(i)}^n} t^{\#\bar{\mathcal{A}}_I(\bar{x})} \\
& = \sum_{\bar{x} \in \mathbb{F}_{\mathrm{p}(i)}^n} t^{\#\overline{\bigsqcup_{i=1}^r \binom{A^{(i)}}{2} \, \sqcup \, \big( A^{(i)} \times \bigsqcup_{j \in R^{(i)}} A^{(j)} \big)}(\bar{x})}\ \text{(Lemma~\ref{LeSA})} \\
& = \sum_{\bar{x} \in \mathbb{F}_{\mathrm{p}(i)}^n} t^{\sum_{s=0}^{\mathrm{p}(i)-1} \sum_{i=1}^r
\binom{\# a_s^{(i)}(\bar{x})}{2}\ +\ \# a_s^{(i)}(\bar{x}) \sum_{j \in R^{(i)}} \# a_s^{(j)}(\bar{x})}\ \text{(Lemma \ref{LePr} (1, 2))} \\
& = \sum_{\substack{a_1^{(1)} + \dots + a_{\mathrm{p}(i)}^{(1)} = \# A^{(1)} \\ \vdots \\ a_1^{(r)} + \dots + a_{\mathrm{p}(i)}^{(r)} = \#A^{(r)}}} \prod_{i=1}^r \binom{\# A^{(i)}}{a_1^{(i)}, \dots, a_{\mathrm{p}(i)}^{(i)}} t^{\sum_{s=1}^{\mathrm{p}(i)} \sum_{i=1}^r \binom{a_s^{(i)}}{2} + a_s^{(i)} \sum_{j \in R^{(i)}} a_s^{(j)}}.
\end{align*}
\end{proof}

\begin{example} The coboundary polynomial of the ideal arrangement $\mathcal{A}_{I_a}$ in Example \ref{ExA} is $\bar{\chi}_{\mathcal{A}_{I_a}}(q,t) = t^{15} + 2qt^{13} + q^2t^{11} - 2t^{13} + 4q^2t^{10} + 8qt^{11} + 6q^3t^8 + 16q^2t^9 - 8qt^{10} - 9t^{11} + 2q^4t^6 + 6q^3t^7 - 14q^2t^8 - 36qt^9 + 4t^{10} + 8q^4t^5 + 45q^3t^6 + 43q^2t^7 + 24qt^8 + 20t^9 + 10q^5t^3 + 64q^4t^4 + 34q^3t^5 - 207q^2t^6 - 144qt^7 - 16t^8 + q^7 + 15q^6t + 75q^5t^2 + 69q^4t^3 - 244q^3t^4 - 146q^2t^5 + 312qt^6 + 95t^7 - 15q^6 - 180q^5t - 701q^4t^2 - 810q^3t^3 + 124q^2t^4 + 90qt^5 - 152t^6 + 95q^5 + 887q^4t + 2585q^3t^2 + 2394q^2t^3 + 376qt^4 + 14t^5 - 329q^4 - 2294q^3t - 4683q^2t^2 - 2856qt^3 - 320t^4 + 672q^3 + 3276q^2t + 4144qt^2 + 1193t^3 - 808q^2 - 2440qt - 1420t^2 + 528q + 736t - 144$, and its Tutte polynomial is
\begin{align*}
T_{\mathcal{A}_{I_a}}(x,y) =\ & x^2y^6+ 2xy^7+ y^8+ x^7 + 2x^4y^3 + 6x^3y^4 + 9x^2y^5 + 10xy^6 + 5y^7 + 8x^6 + 10x^5y \\
& + 14x^4y^2 + 22x^3y^3 + 33x^2y^4 + 32xy^5 + 15y^6 + 26x^5 + 50x^4y + 73x^3y^2 + 83x^2y^3 \\
& + 68xy^4 + 29y^5 + 44x^4 + 94x^3y + 120x^2y^2 + 96xy^3 + 38y^4 + 41x^3 + 82x^2y \\
& + 77xy^2 + 32y^3 + 20x^2 + 32xy + 16y^2 + 4x + 4y.
\end{align*}
\end{example}

\begin{theorem}
Let $\mathcal{A}_I$ be a full connected ideal arrangement of $\Phi_{B_n}$, with associated partition $A^{(1)}| \dots |A^{(r)}|B^{(1)}| \dots |B^{(s)}$, and $R^{(u)} = \big\{l \in \{u+1, \dots, r\}\ |\ s_I(A^{(u)}) \cap s_I(A^{(l)}) \neq \emptyset \big\}$,\\ $R_A^{(v)} = \big\{l \in [r]\ |\ s_I(B^{(v)}) \cap s_I(A^{(l)}) \neq \emptyset \big\}$, $S^{(v)} = \big\{h \in [v-1]\ |\ s_I(B^{(v)}) \cap s_I(B^{(h)}) \neq \emptyset \big\}$, $R_0 = \big\{l \in [r]\ |\ s_I(0) \cap s_I(A^{(l)}) \neq \emptyset \big\}$, and $S_0 = \big\{h \in [s]\ |\ s_I(0) \cap s_I(B^{(h)}) \neq \emptyset \big\}$.\\
Then, for a positive integer $i$, we have
$$\bar{\chi}_{\mathcal{A}_I}\big(\mathrm{p}(i),t\big) = \sum_{\substack{a_0^{(1)} + \dots + a_{\mathrm{p}(i)-1}^{(1)} = \# A^{(1)} \\ \vdots \\ a_0^{(r)} + \dots + a_{\mathrm{p}(i)-1}^{(r)} = \# A^{(r)} \\ b_0^{(1)} + \dots + b_{\mathrm{p}(i)-1}^{(1)} = \# B^{(1)} \\ \vdots \\ b_0^{(s)} + \dots + b_{\mathrm{p}(i)-1}^{(s)} = \# B^{(s)}}} \prod_{u=1}^r \binom{\# A^{(u)}}{a_0^{(u)}, \dots, a_{\mathrm{p}(i)-1}^{(u)}} \prod_{v=1}^s \binom{\# B^{(v)}}{b_0^{(v)}, \dots, b_{\mathrm{p}(i)-1}^{(v)}}t^{\mathrm{f}_B(u,v)},$$
\begin{align*}
\text{with}\ \mathrm{f}_B(u,v) = & \sum_{q=0}^{\mathrm{p}(i)-1} \Big( \binom{a_q^{(u)}}{2} + a_q^{(u)} \sum_{l \in R^{(u)}} a_q^{(l)} \Big) + 2 b_0^{(v)} \Big( \frac{b_0^{(v)}-1}{2} + \sum_{l \in R_A^{(v)}} a_0^{(l)} + \sum_{h \in S^{(v)}} b_0^{(h)} \Big) \\
& + \sum_{q=1}^{\mathrm{p}(i)-1} b_q^{(v)} \Big( \frac{b_q^{(v)}-1}{2} + \sum_{l \in R_A^{(v)}} \big( a_q^{(l)} + a_{\mathrm{p}(i)-q}^{(l)} \big) + \sum_{h \in S^{(v)}} \big( b_q^{(h)} + b_{\mathrm{p}(i)-q}^{(h)} \big) \Big) \\
& + \sum_{q=1}^{\frac{\mathrm{p}(i)-1}{2}} b_q^{(v)} \times b_{\mathrm{p}(i)-q}^{(v)} + \sum_{l \in R_0} a_0^{(l)} + \sum_{h \in S_0} b_0^{(h)}.
\end{align*}
\end{theorem}

\begin{proof}
We have $\dim \cap \mathcal{A}_I = 0$ for every full connected ideal arrangement $\mathcal{A}_I$ of $\Phi_{B_n}$. The proof is similar to that of Theorem \ref{PrA}, using Lemma~\ref{LeB} and Lemma~\ref{LePr} (1, 2, 3 ,4, 5).
\end{proof}

\begin{example} The coboundary polynomial of the ideal arrangement $\mathcal{A}_{I_b}$ in Example \ref{ExB} is $\bar{\chi}_{\mathcal{A}_{I_b}}(q,t) = t^{21} + qt^{18} - t^18 + 2qt^{15} + 2q^2t^{13} + 4qt^{14} - 2t^{15} - 4qt^{13} - 4t^{14} + 3q^2t^{11} + 2qt^{12} + 2t^{13} + q^3t^9 + 5q^2t^{10} - 4qt^{11} - 2t^{12} + q^2t^9 - 2qt^{10} + t^{11} + 6q^3t^7 + 24q^2t^8 + 4qt^9 - 3t^{10} + 9q^3t^6 + 2q^2t^7 - 40qt^8 - 6t^9 + 3q^4t^4 + 33q^3t^5 + 50q^2t^6 - 14qt^7 + 16t^8 + 23q^4t^3 + 123q^3t^4 - 18q^2t^5 - 225qt^6 + 6t^7 + q^6 + 21q^5t + 123q^4t^2 - q^3t^3 - 783q^2t^4 - 290qt^5 + 166t^6 - 21q^5 - 327q^4t - 1362q^3t^2 - 853q^2t^3 + 1427qt^4 + 275t^5 + 178q^4 + 1965q^3t + 5391q^2t^2 + 2508qt^3 - 770t^4 - 774q^3 - 5625q^2t - 8778qt^2 - 1677t^3 + 1801q^2 + 7494qt + 4626t^2 - 2085q - 3528t + 900$, and its Tutte polynomial is 
\begin{align*}
T_{\mathcal{A}_{I_b}}(x,y) =\ & y^{15} + xy^{13} + 6y^{14} + 5xy^{12} + 20y^{13} + 15xy^{11} + 50y^{12} + 2x^2y^9 + 37xy^{10} + 105y^{11} \\
& + 8x^2y^8 + 80xy^9 + 194y^{10} + x^3y^6 + 23x^2y^7 + 156xy^8 + 322y^9 + 3x^3y^5 + 54x^2y^6 \\
& + 276xy^7 + 486y^8 + 12x^3y^4 + 112x^2y^5 + 445xy^6 + 672y^7 + x^6 + 3x^4y^2 + 37x^3y^3 \\
& + 214x^2y^4 + 662xy^5 + 854y^6 + 15x^5 + 29x^4y + 99x^3y^2 + 370x^2y^3 + 899xy^4 \\
& + 989y^5 + 88x^4 + 241x^3y + 586x^2y^2 + 1096xy^3 + 1021y^4 + 252x^3 + 682x^2y \\
& + 1102xy^2 + 888y^3 + 352x^2 + 728xy + 568y^2 + 192x + 192y.
\end{align*}
\end{example}

\begin{theorem}
Let $\mathcal{A}_I$ be a full connected ideal arrangement of $\Phi_{C_n}$, with associated partition $A^{(1)}| \dots |A^{(r)}|B^{(1)}| \dots |B^{(s)}$, and $R^{(u)} = \big\{l \in \{u+1, \dots, r\}\ |\ s_I(A^{(u)}) \cap s_I(A^{(l)}) \neq \emptyset \big\}$,\\ $R_A^{(v)} = \big\{l \in [r]\ |\ s_I(B^{(v)}) \cap s_I(A^{(l)}) \neq \emptyset \big\}$, $S^{(v)} = \big\{h \in [v-1]\ |\ s_I(B^{(v)}) \cap s_I(B^{(h)}) \neq \emptyset \big\}$, $R_0 = \big\{l \in [r]\ |\ s_I(0) \cap s_I(A^{(l)}) \neq \emptyset \big\}$, and $S_0 = \big\{h \in [s]\ |\ s_I(0) \cap s_I(B^{(h)}) \neq \emptyset \big\}$.\\
Then, for a positive integer $i$, we have
$$\bar{\chi}_{\mathcal{A}_I}\big(\mathrm{p}(i),t\big) = \sum_{\substack{a_0^{(1)} + \dots + a_{\mathrm{p}(i)-1}^{(1)} = \# A^{(1)} \\ \vdots \\ a_0^{(r)} + \dots + a_{\mathrm{p}(i)-1}^{(r)} = \# A^{(r)} \\ b_0^{(1)} + \dots + b_{\mathrm{p}(i)-1}^{(1)} = \# B^{(1)} \\ \vdots \\ b_0^{(s)} + \dots + b_{\mathrm{p}(i)-1}^{(s)} = \# B^{(s)}}} \prod_{u=1}^r \binom{\# A^{(u)}}{a_0^{(u)}, \dots, a_{\mathrm{p}(i)-1}^{(u)}} \prod_{v=1}^s \binom{\# B^{(v)}}{b_0^{(v)}, \dots, b_{\mathrm{p}(i)-1}^{(v)}}t^{\mathrm{f}_B(u,v)},$$
\begin{align*}
\text{with}\ \mathrm{f}_B(u,v) = & \sum_{q=0}^{\mathrm{p}(i)-1} \Big( \binom{a_q^{(u)}}{2} + a_q^{(u)} \sum_{l \in R^{(u)}} a_q^{(l)} \Big) + 2 b_0^{(v)} \Big( \frac{b_0^{(v)}-1}{2} + \sum_{l \in R_A^{(v)}} a_0^{(l)} + \sum_{h \in S^{(v)}} b_0^{(h)} \Big) \\
& + \sum_{q=1}^{\mathrm{p}(i)-1} b_q^{(v)} \Big( \frac{b_q^{(v)}-1}{2} + \sum_{l \in R_A^{(v)}} \big( a_q^{(l)} + a_{\mathrm{p}(i)-q}^{(l)} \big) + \sum_{h \in S^{(v)}} \big( b_q^{(h)} + b_{\mathrm{p}(i)-q}^{(h)} \big) \Big) \\
& + \sum_{q=1}^{\frac{\mathrm{p}(i)-1}{2}} b_q^{(v)} \times b_{\mathrm{p}(i)-q}^{(v)} + \sum_{l \in R_0} a_0^{(l)} + \sum_{h \in S_0} b_0^{(h)}.
\end{align*}
\end{theorem}

\begin{proof}
We have $\dim \cap \mathcal{A}_I = 0$ for every full connected ideal arrangement $\mathcal{A}_I$ of $\Phi_{C_n}$. The proof is similar to that of Theorem \ref{PrA}, using Lemma~\ref{LeC} and Lemma~\ref{LePr} (1, 2, 3 ,4, 5).
\end{proof}

\begin{example} The coboundary polynomial of the ideal arrangement $\mathcal{A}_{I_c}$ in Example \ref{ExC} is $\bar{\chi}_{\mathcal{A}_{I_c}}(q,t) = t^{21} + qt^{18} - t^{18} + 2qt^{15} + 2q^2t^{13} + 3qt^{14} - 2t^{15} - 3qt^{13} - 3t^{14} + 2q^2t^{11} + 3qt^{12} + t^{13} + q^3t^9 + 6q^2t^{10} - 2qt^{11} - 3t^{12} - 6qt^{10} + 4q^3t^7 + 25q^2t^8 + 16qt^9 + 13q^3t^6 + 25q^2t^7 - 52qt^8 - 17t^9 + 3q^4t^4 + 37q^3t^5 + 17q^2t^6 - 80qt^7 + 27t^8 + 23q^4t^3 + 107q^3t^4 - 73q^2t^5 - 152qt^6 + 51t^7 + q^6 + 21q^5t + 123q^4t^2 + 13q^3t^3 - 636q^2t^4 - 111qt^5 + 122t^6 - 21q^5 - 327q^4t - 1366q^3t^2 - 963q^2t^3 + 1046qt^4 + 147t^5 + 178q^4 + 1965q^3t + 5419q^2t^2 + 2764qt^3 - 520t^4 - 774q^3 - 5625q^2t - 8838qt^2 - 1837t^3 + 1801q^2 + 7494qt + 4662t^2 - 2085q - 3528t + 900$, and its Tutte polynomial is
\begin{align*}
T_{\mathcal{A}_{I_c}}(x,y) =\ & y^{15} + xy^{13} + 6y^{14} + 5xy^{12} + 20y^{13} + 15xy^{11} + 50y^{12} + 2x^2y^9 + 37xy^{10} + 105y^{11} \\
& + 8x^2y^8 + 79xy^9 + 194y^{10} + x^3y^6 + 22x^2y^7 + 152xy^8 + 323y^9 + 3x^3y^5 + 51x^2y^6 \\
& + 269xy^7 + 491y^8 + 10x^3y^4 + 105x^2y^5 + 438xy^6 + 685y^7 + x^6 + 3x^4y^2 + 35x^3y^3 \\
& + 207x^2y^4 + 662xy^5 + 878y^6 + 15x^5 + 29x^4y + 103x^3y^2 + 378x^2y^3 + 920xy^4 \\
& + 1024y^5 + 88x^4 + 241x^3y + 602x^2y^2 + 1130xy^3 + 1055y^4 + 252x^3 + 682x^2y \\
& + 1118xy^2 + 904y^3 + 352x^2 + 728xy + 568y^2 + 192x + 192y.
\end{align*}
\end{example}

\begin{theorem}
Let $\mathcal{A}_I$ be a full connected ideal arrangement of $\Phi_{D_n}$, with associated partition $A^{(1)}| \dots |A^{(r)}|B^{(1)}| \dots |B^{(s)}$, and let $R^{(u)} = \big\{l \in \{u+1, \dots, r\}\ |\ s_I(A^{(u)}) \cap s_I(A^{(l)}) \neq \emptyset \big\}$, $R_A^{(v)} = \big\{l \in [r]\ |\ s_I(B^{(v)}) \cap s_I(A^{(l)}) \neq \emptyset \big\}$, and $S^{(v)} = \big\{h \in [v-1]\ |\ s_I(B^{(v)}) \cap s_I(B^{(h)}) \neq \emptyset \big\}$.\\
Then, for a positive integer $i$, we have
$$\bar{\chi}_{\mathcal{A}_I}\big(\mathrm{p}(i),t\big) = \sum_{\substack{a_0^{(1)} + \dots + a_{\mathrm{p}(i)-1}^{(1)} = \# A^{(1)} \\ \vdots \\ a_0^{(r)} + \dots + a_{\mathrm{p}(i)-1}^{(r)} = \# A^{(r)} \\ b_0^{(1)} + \dots + b_{\mathrm{p}(i)-1}^{(1)} = \# B^{(1)} \\ \vdots \\ b_0^{(s)} + \dots + b_{\mathrm{p}(i)-1}^{(s)} = \# B^{(s)}}} \prod_{u=1}^r \binom{\# A^{(u)}}{a_0^{(u)}, \dots, a_{\mathrm{p}(i)-1}^{(u)}} \prod_{v=1}^s \binom{\# B^{(v)}}{b_0^{(v)}, \dots, b_{\mathrm{p}(i)-1}^{(v)}}t^{\mathrm{f}_D(u,v)},$$
\begin{align*}
\text{with}\ \mathrm{f}_D(u,v) = & \sum_{q=0}^{\mathrm{p}(i)-1} \Big( \binom{a_q^{(u)}}{2} + a_q^{(u)} \sum_{l \in R^{(u)}} a_q^{(l)} \Big) + 2 b_0^{(v)} \Big( \frac{b_0^{(v)}-1}{2} + \sum_{l \in R_A^{(v)}} a_0^{(l)} + \sum_{h \in S^{(v)}} b_0^{(h)} \Big) \\
& + \sum_{q=1}^{\mathrm{p}(i)-1} b_q^{(v)} \Big( \frac{b_q^{(v)}-1}{2} + \sum_{l \in R_A^{(v)}} \big( a_q^{(l)} + a_{\mathrm{p}(i)-q}^{(l)} \big) + \sum_{h \in S^{(v)}} \big( b_q^{(h)} + b_{\mathrm{p}(i)-q}^{(h)} \big) \Big) \\
& + \sum_{q=1}^{\frac{\mathrm{p}(i)-1}{2}} b_q^{(v)} \times b_{\mathrm{p}(i)-q}^{(v)}.
\end{align*}
\end{theorem}

\begin{proof}
We have $\dim \cap \mathcal{A}_I = 0$, for every full connected ideal arrangement $\mathcal{A}_I$ of $\Phi_{D_n}$. The proof is similar to that of Proposition \ref{PrA}, using Lemma~\ref{LeD} and Lemma~\ref{LePr} (1, 2, 3, 4).
\end{proof}

\begin{example} The coboundary polynomial of the ideal arrangement $\mathcal{A}_{I_d}$ in Example \ref{ExD} is $\bar{\chi}_{\mathcal{A}_{I_d}}(q,t) = t^{16} + qt^{14} - t^{14} + qt^{12} + 3q^2t^{10} + 4qt^{11} - t^{12} - 2qt^{10} - 4t^{11} + q^3t^7 + 11q^2t^8 + 2qt^9 - t^{10} + 5q^3t^6 + q^2t^7 - 23qt^8 - 2t^9 + 13q^3t^5 + 35q^2t^6 + 21qt^7 + 12t^8 + 16q^4t^3 + 72q^3t^4 - 19q^2t^5 - 134qt^6 - 23t^7 + q^6 + 16q^5t + 72q^4t^2 - 9q^3t^3 - 292q^2t^4 - 19qt^5 + 94t^6 - 16q^5 - 192q^4t - 631q^3t^2 - 353q^2t^3 + 332qt^4 + 25t^5 + 104q^4 + 899q^3t + 2012q^2t^2 + 923qt^3 - 112t^4 - 350q^3 - 2037q^2t - 2717qt^2 - 577t^3 + 639q^2 + 2205qt + 1264t^2 - 594q - 891t + 216$, and its Tutte polynomial is
\begin{align*}
T_{\mathcal{A}_{I_d}}(x,y) =\ & xy^9 + y^{10} + 5xy^8 + 5y^9 + 3x^2y^6 + 16xy^7 + 15y^8 + x^3y^4 + 12x^2y^5 + 38xy^6 + 34y^7 \\
& + x^6 + 8x^3y^3 + 38x^2y^4 + 79xy^5 + 63y^6 + 10x^5 + 16x^4y + 34x^3y^2 + 81x^2y^3 \\
& + 134xy^4 + 95y^5 + 39x^4 + 87x^3y + 152x^2y^2 + 191xy^3 + 117y^4 + 74x^3 \\
& + 162x^2y + 196xy^2 + 112y^3 + 68x^2 + 116xy + 72y^2 + 24x + 24y.
\end{align*}
\end{example}

\section{Exceptional Root Systems}  \label{SeEx}

\noindent We introduce a linear order on the exceptional root systems $\Phi_{G_2}^+$, $\Phi_{F_4}^+$, and $\Phi_{E_6}^+$, and expose the formula of Crapo by means of this order. As the formula of Crapo computes the Tutte polynomial of a vector set, we draw the Hasse diagram of these root systems in order to visualize the vectors that make up their ideals. Then, we compute some examples of Tutte polynomials of ideal arrangements. These computings are done with \texttt{SageMath}.

\smallskip

\noindent Take an exceptional root system $\Phi_{X_n}$, $X_n \in \{G_2, F_4, E_6\}$, associated to a simple system $\Delta_{X_n} = \{\alpha_1, \dots, \alpha_n\}$. Define the function $\mathrm{l}: \Phi_{X_n}^+ \rightarrow \mathbb{N}^*$ by
$$\text{if}\ u = \sum_{i=1}^n u_i \alpha_i,\ \text{then}\ \mathrm{l}(u) := \overbrace{1 \dots 1}^{u_1\, \text{times}}\, \dots \, \overbrace{n \dots n}^{u_n\, \text{times}}.$$
It is clear that $\mathrm{l}$ is a bijection between $\Phi_{X_n}^+$ and $\mathrm{l}(\Phi_{X_n}^+)$. Define the linear order $\lhd$ on $\Phi_{X_n}^+$ by
$$\forall a, b \in \Phi_{X_n}^+:\ a \lhd b \, \Leftrightarrow \, \mathrm{l}(a) < \mathrm{l}(b).$$

\noindent Let $\mathrm{r}$ be the rank function of vector sets in $\mathbb{R}^n$, and $X$ a subset of $\Phi_{X_n}^+$. A basis of $X$ is a subset $B$ of $X$ such that $\mathrm{r}(B) = |B| = \mathrm{r}(X)$. Denote by $\mathscr{B}(X)$ the basis set of $X$.

\noindent For a subset $A$ of $X$, and an element $x$ in $X$, define the set
$$A_{\lhd x} := \{a \in A\ |\ a \lhd x\}.$$

\noindent Let $X$ be a subset of $\Phi_{X_n}^+$, and take a basis $B$ in $\mathscr{B}(X)$:
\begin{itemize}
\item Let $b \in B$. One says that $b$ is an internal active element of $B$ if $$\forall x \in X_{\lhd b} \setminus B:\ \mathrm{r}\big(\{x\} \sqcup (B \setminus \{b\})\big) < n.$$
\item Let $x \in X \setminus B$. One says that $x$ is an external active element of $B$ if
$$\mathrm{r}\big(\{x\} \sqcup B_{\rhd x}\big) = \mathrm{r}(B_{\rhd x}).$$
\end{itemize}

\noindent Denote by $i(B)$ resp. $e(B)$ the number of internal resp. external active elements of a basis $B$. We compute the Tutte polynomial of the hyperplane arrangement $\mathcal{A} = \{x^{\perp}\}_{x \in X}$ by using the formula of Crapo \cite[Theorem~2.32]{DePr} $$T_{\mathcal{A}}(x,y) = \sum_{B \in \mathscr{B}(X)} x^{i(B)} y^{e(B)}.$$

\noindent In our case, $X$ is a complement $\Phi_{X_n}^+ \setminus I$ of an ideal $I$ of $\Phi_{X_n}^+$. We represent the Hasse diagram of $(\Phi_{G_2}^+, \preceq)$ resp. $(\Phi_{F_4}^+, \preceq)$ resp. $(\Phi_{E_6}^+, \preceq)$ in Figure \ref{G2} resp. \ref{F4} resp. \ref{E6}. In the Hasse diagrams, a vector $u$ of $\Phi_{X_n}^+$ is represented by $X\mathrm{l}(u)$.

\begin{example}
$G1112$ is the vector $(3,1)$, $F1234$ is the vector $(1,1,1,1)$, and $E123445$ is the vector $(1,1,1,2,1,0)$. 
\end{example}

\begin{figure}[h]
\centering
\includegraphics[scale=0.8]{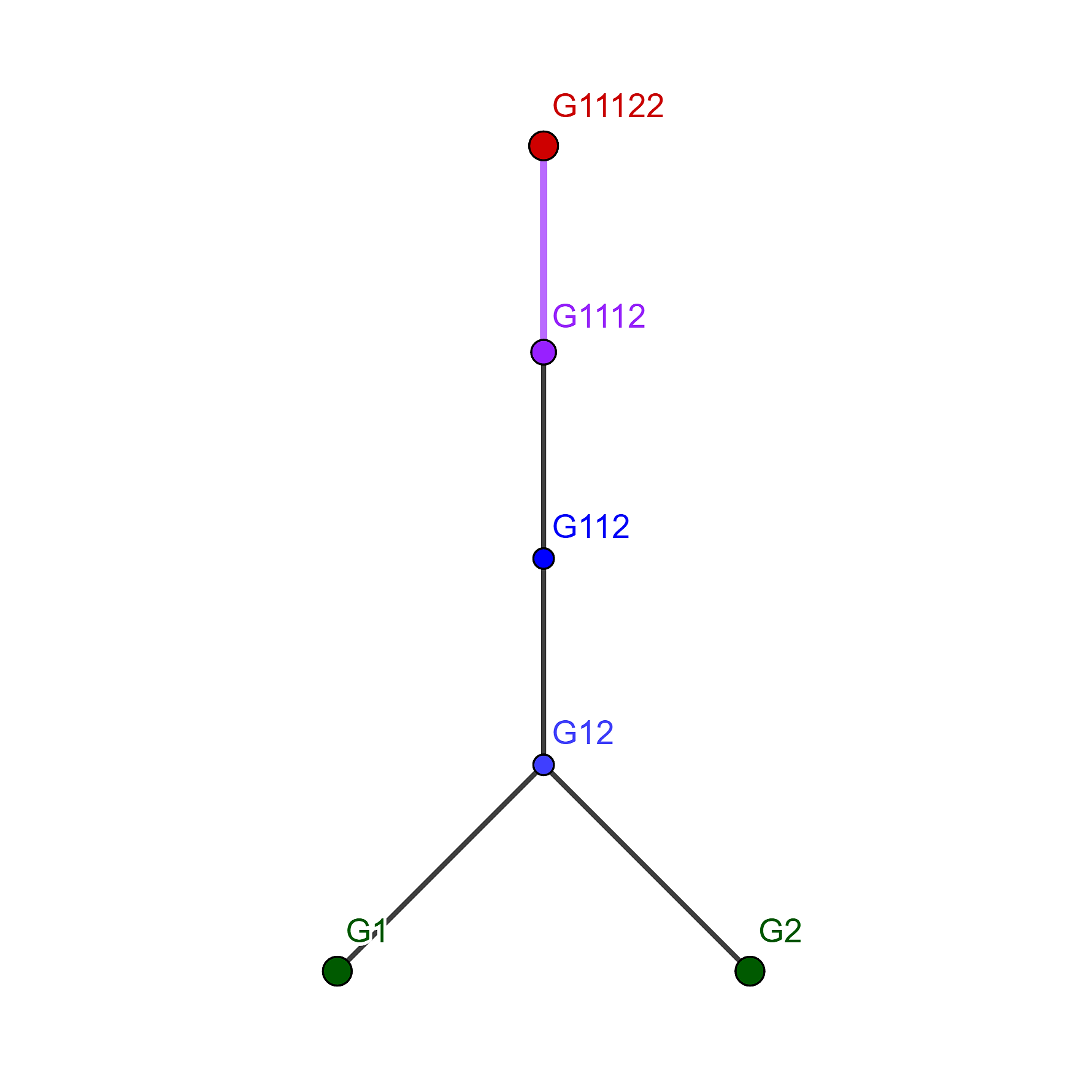}
\caption{Hasse Diagram of $(\Phi_{G_2}^+, \preceq)$}
\label{G2}
\end{figure}

\begin{figure}[h]
\centering
\includegraphics[scale=0.8]{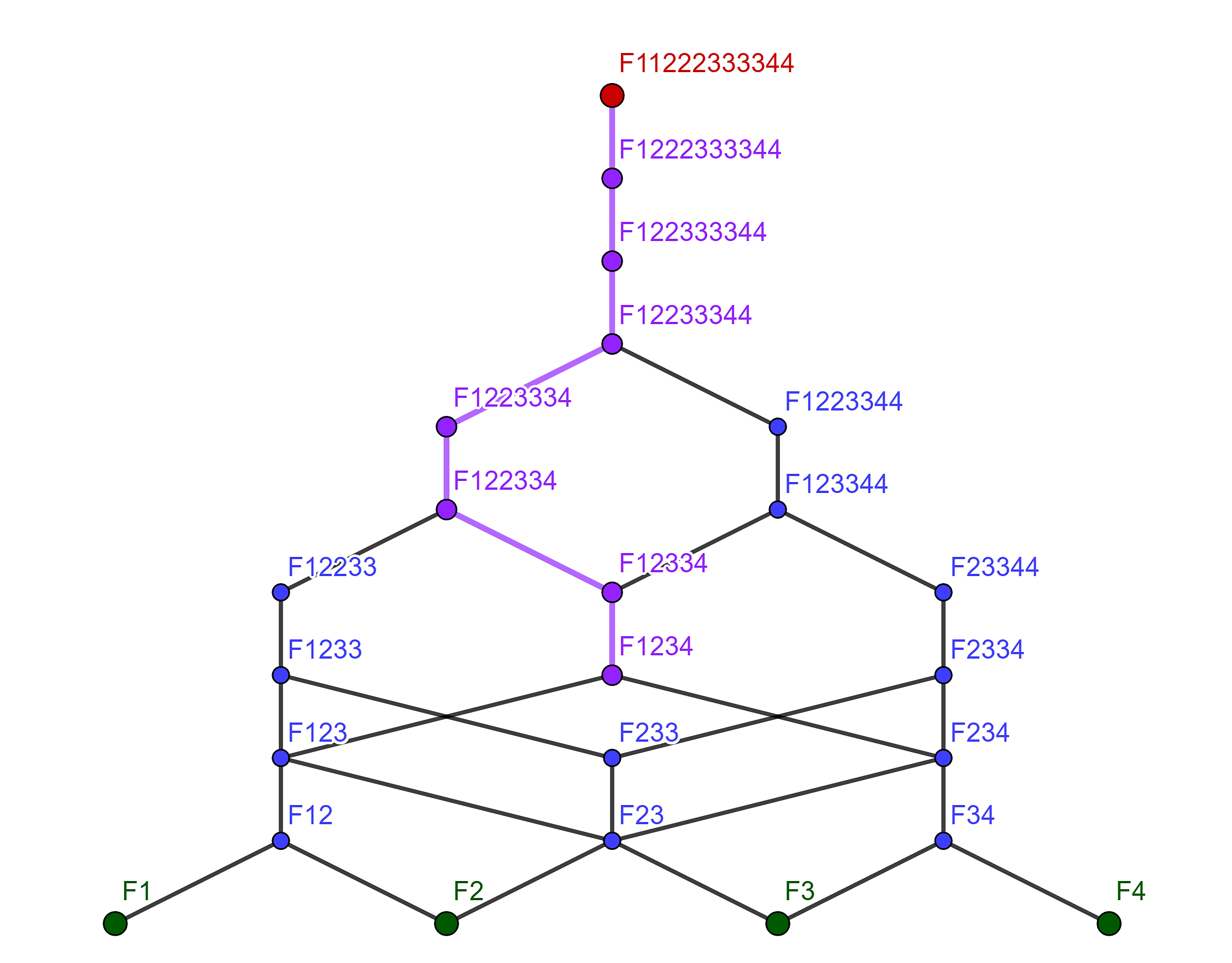}
\caption{Hasse Diagram of $(\Phi_{F_4}^+, \preceq)$}
\label{F4}
\end{figure}

\begin{figure}[h]
\centering
\includegraphics[scale=0.8]{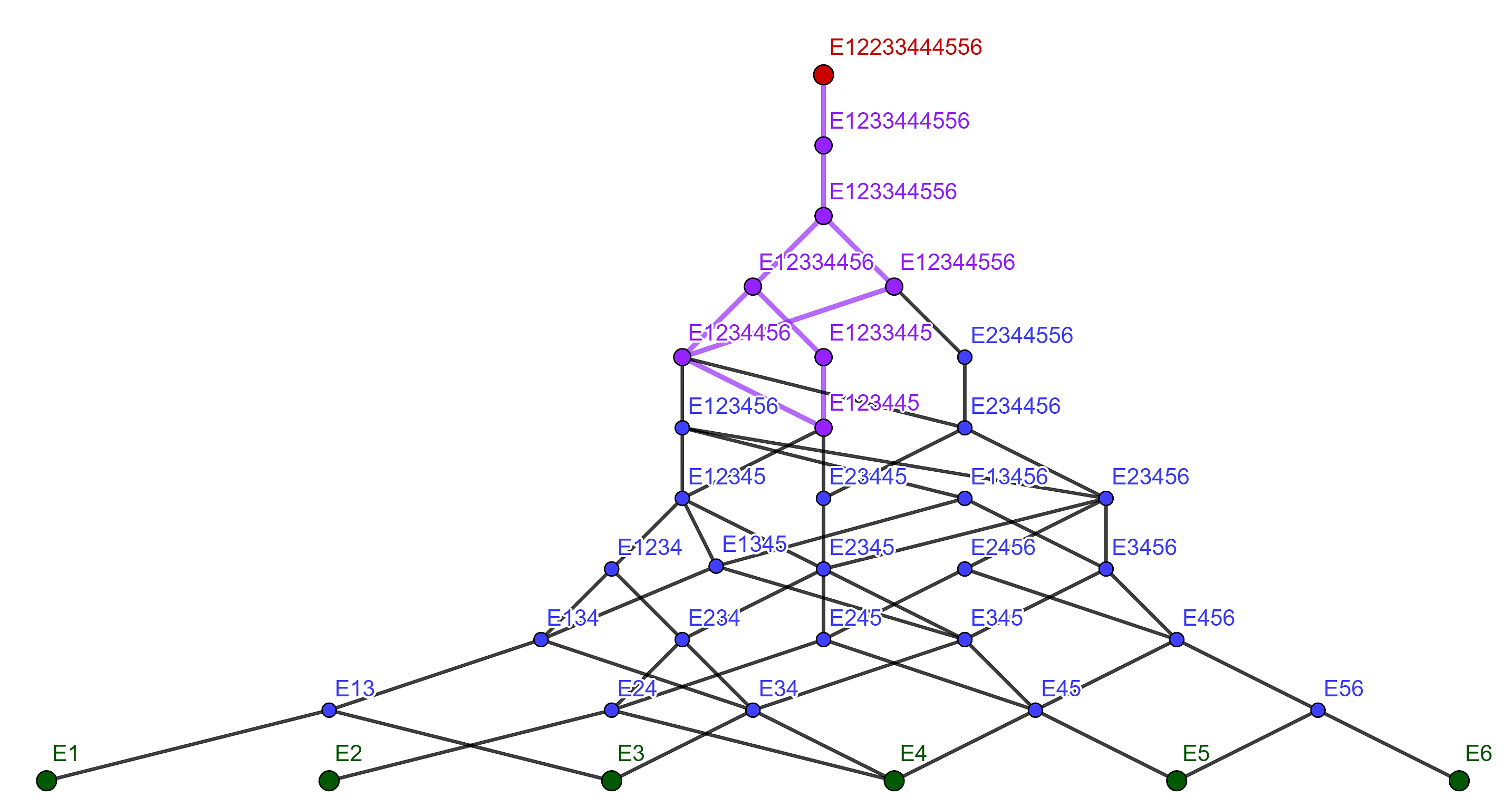}
\caption{Hasse Diagram of $(\Phi_{E_6}^+, \preceq)$}
\label{E6}
\end{figure}

\noindent \emph{An ideal $I$ of $\Phi_{X_n}^+$ is a connected graph in the Hasse diagram of $(\Phi_{X_n}^+, \preceq)$ containing its maximal element.} We compute the following Tutte polynomials with the formula of Crapo.

\begin{example}
The vector tuple $I_g = \big((3,1), (3,2)\big)$ is an ideal of $\Phi_{G_2}^+$, and the Tutte polynomial of its associated hyperplane arrangement is $T_{\mathcal{A}_{I_g}}(x,y) = x^2 + y^2 + 2x + 2y$.
\end{example}

\begin{example}
The vector tuple\\ $I_f = \big((1,1,1,1), (1,1,2,1), (1,2,2,1), (1,2,3,1), (1,2,3,2), (1,2,4,2), (1,3,4,2), (2,3,4,2)\big)$\\ is an ideal of $\Phi_{F_4}^+$, and the Tutte polynomial of its associated hyperplane arrangement is
\begin{align*}
T_{\mathcal{A}_{I_f}}(x,y) =\ & y^{12} + 4y^{11} + 10y^{10} + 20y^9 + 35y^8 + 2xy^6 + 56y^7 + 7xy^5 + 82y^6 + 19xy^4 + 111y^5 \\ & + x^4 + 5x^2y^2 + 45xy^3 + 137y^4 + 12x^3 + 25x^2y + 83xy^2 + 147y^3 \\ &
+ 48x^2 + 109xy + 125y^2 + 64x + 64y. 
\end{align*}
\end{example}

\begin{example}
The vector tuple
\begin{align*}
I_e =\ & \big((1,1,1,2,1,0), (1,1,1,2,1,1), (1,1,2,2,1,0), (1,1,2,2,1,1), (1,1,1,2,2,1), (1,1,2,2,2,1), \\
& (1,1,2,3,2,1), (1,2,2,3,2,1)\big)  
\end{align*}
is an ideal of $\Phi_{E_6}^+$, and the Tutte polynomial of its associated hyperplane arrangement is
\begin{align*}
T_{\mathcal{A}_{I_e}}(x,y) =\ & y^{22} + 6y^{21} + 21y^{20} + 56y^{19} + 126y^{18} + 252y^{17} + xy^{15} + 462y^{16} + 5xy^{14} + 791y^{15} \\
& + 18xy^{13} + 1281y^{14} + 52xy^{12} + 1978y^{13} + 129xy^{11} + 2927y^{12} + 295xy^{10} + 4163y^{11} \\
& + 5x^2y^8 + 623xy^9 + 5688y^{10} + 26x^2y^7 + 1212xy^8 + 7445y^9 \\
& + 110x^2y^6 + 2176xy^7 + 9288y^8 + 346x^2y^5 + 3596xy^6 + 10957y^7 \\
& + x^6 + 79x^3y^3 + 892x^2y^4 + 5404xy^5 + 12065y^6 \\
& + 22x^5 + 62x^4y + 303x^3y^2 + 1829x^2y^3 + 7235xy^4 + 12159y^5 \\
& + 191x^4 + 762x^3y + 2863x^2y^2 + 8292xy^3 + 10860y^4 \\
& + 818x^3 + 3184x^2y + 7646xy^2 + 8136y^3 + 1728x^2 + 4872xy + 4584y^2 + 1440x + 1440y. 
\end{align*}
\end{example}

\newpage

\bibliographystyle{abbrvnat}

\end{document}